\documentclass[12pt]{amsart}

\usepackage{verbatim}
\usepackage{amssymb}
\usepackage[alphabetic]{amsrefs}
\usepackage[pdftex]{graphicx}
\usepackage[all]{xy}
\usepackage{color}

\newtheorem{theorem}{Theorem}[section]
\newtheorem{corollary}[theorem]{Corollary}
\newtheorem{lemma}[theorem]{Lemma}
\newtheorem{proposition}[theorem]{Proposition}

\newtheorem{problem}[theorem]{Problem}
\newtheorem{definition/theorem}[theorem]{Definition/Theorem}
\newtheorem{definition/proposition}[theorem]{Definition/Proposition}

\theoremstyle{definition} \newtheorem{definition}[theorem]{Definition}
\newtheorem{remark}[theorem]{Remark}
\newtheorem{example}[theorem]{Example}

\newtheorem{question}[theorem]{Question}

\newtheorem{hypothesis}[theorem]{Hypothesis}

\newcommand{\K}{\Bbbk}
\newcommand{\MOn}{\overline{M}_{0,n}}
\newcommand{\git}{\ensuremath{\operatorname{/\!\!/}}}

\newcommand{\G}{\ensuremath{\mathcal G_{\Delta}}}
\newcommand{\Gprime}{\ensuremath{\mathcal G_{\Delta^{\circ}}}}

\renewcommand{\L}{\ensuremath{\mathcal L_{\Delta}}}
\newcommand{\U}{\ensuremath{\mathcal F_{\Delta}}}
\newcommand{\Uw}{\ensuremath{\mathcal F_{\Delta,\mathbf{w}}}}

\DeclareMathOperator{\intt}{int}
\DeclareMathOperator{\relint}{relint}
\DeclareMathOperator{\trop}{trop}

\DeclareMathOperator{\inn}{in}
\DeclareMathOperator{\Pic}{Pic}

\DeclareMathOperator{\pos}{pos}

\DeclareMathOperator{\Hom}{Hom}
\DeclareMathOperator{\Nef}{Nef}
\DeclareMathOperator{\nef}{Nef}

\DeclareMathOperator{\Proj}{Proj}

\DeclareMathOperator{\Cox}{Cox}

\DeclareMathOperator{\val}{val}

\DeclareMathOperator{\starr}{star}

\DeclareMathOperator{\divv}{div}
\DeclareMathOperator{\Eff}{Eff}
 
\DeclareMathOperator{\Ass}{Ass}
\DeclareMathOperator{\Bl}{Bl}

\begin{document}

\title{Lower and upper bounds for nef cones}
%\shorttitle{Lower and upper bounds for nef cones}

\author{Angela Gibney}

\address{Department of Mathematics\\
University of Georgia\\
Athens, GA 30602\\
USA
}

\email{agibney@math.uga.edu}

\author{Diane Maclagan}
\address{Mathematics Institute\\
Zeeman Building\\
University of Warwick\\
Coventry CV4 7AL\\
United Kingdom}

\email{D.Maclagan@warwick.ac.uk}
\date{\today}

\maketitle

\begin{abstract}The nef cone of a projective variety $Y$ is an important and often elusive invariant.
In this paper we construct two polyhedral lower bounds and one polyhedral upper bound for the nef cone of $Y$ using an embedding of $Y$ into a toric variety. 
The lower bounds generalize the combinatorial description of the nef
cone of a Mori dream space, while the upper bound generalizes the
$\operatorname{F}$-conjecture for the nef cone of the moduli space
$\MOn$ to a wide class of varieties.
\end{abstract}

\date{\today}

\maketitle

\section{Introduction}
A central goal of birational geometry is to understand maps between
projective varieties.  The cone $\Nef(Y)$ of divisors that
nonnegatively intersect all curves on $Y$ encodes information about
the possible morphisms from $Y$ to other projective varieties.  The
interior of this cone is the cone of ample divisors, multiples of
which give rise to projective embeddings of $Y$, while divisors on the
boundary of the cone determine other morphisms.
This cone is hard to compute in general,
and is unknown for many even elementary varieties.  One contributing
factor to our general ignorance is that these cones, while convex, can
be very complicated, and in particular need not be polyhedral.

In this paper we construct polyhedral upper and lower bounds for nef
cones of varieties.  This gives (separate) necessary and sufficient
conditions for a divisor to be nef: the lower bound is a polyhedral
cone whose interior consists of divisors we certify to be ample, while
if a divisor lives outside the polyhedral upper bound cone it is
definitely not ample.  Our program exploits well-chosen embeddings of
the variety $Y$ into a toric variety $X_{\Delta}$, and unifies several
different approaches found in the literature.

Equivalent recipes for the nef cone of a projective toric variety
$X_{\Delta}$ can give rise to different cones in
$\Pic(X_{\Delta})_{\mathbb{R}}:= \Pic(X_{\Delta}) \otimes \mathbb R$
when the variety is not complete.  Explicitly, on a projective toric
variety, a divisor is nef if and only if it is globally generated, if
and only if its pullback to every torus invariant subvariety is
effective, and if and only if it nonnegatively intersects every torus
invariant curve.  When $X_{\Delta}$ is not projective, the three
corresponding cones $\G$, $\L$, and $\Uw$ in
$\Pic(X_{\Delta})_{\mathbb{R}}$ satisfy $\G \subseteq \L \subseteq
\Uw$, where each inclusion can be proper; see
Proposition~\ref{p:toriccontainment}.  The vector $\mathbf{w}$ is a
cohomological invariant that compensates for the fact that a
non-projective toric variety may have no torus-invariant curves; see
Definition~\ref{d:Fdef}.

Given an embedding $i: Y \rightarrow X_{\Delta}$, we pull these three
cones in $\Pic(X_{\Delta})_{\mathbb R}$ back to the N\'eron Severi space $N^1(Y)_{\mathbb
  R}$ of $Y$.  In this way, for sufficiently general embeddings, we
obtain (Theorem~\ref{t:bound}) both lower and upper bounds
$\mathcal{G}_{\Delta}(Y) \subseteq \mathcal{L}_{\Delta}(Y) \subseteq
\nef(Y) \subseteq \mathcal{F}_{\Delta}(Y)$ for the nef cone of $Y$.
The lower bounds hold for any toric embedding, while the upper bound
requires that the induced map $i^* \colon \Pic(X_{\Delta})_{\mathbb R}
 \rightarrow N^1(Y)_{\mathbb R}$ is surjective, and that the
fan $\Delta$ equals the tropical variety of $Y \cap T$, where $T$ is
the torus of $X_{\Delta}$; see Section~\ref{bounds}.

The requirement of a toric embedding for $Y$ does not impose any
restrictions on $Y$; every projective variety embeds into the toric
variety $\mathbb P^N$.  The pullback of $\mathcal O(1)$ on $\mathbb
P^N$ can be considered a (not-very-informative) lower bound for
$\nef(Y)$.  Toric embeddings can be chosen 
so that the resulting lower bound is a
full-dimensional subcone of the nef cone.  Our bounds depend on the
choice of toric embedding, and a given variety may have several useful
embeddings.  In addition, many interesting varieties come with natural
embeddings into toric varieties satisfying all required conditions;
see Sections~\ref{s:MDS} and \ref{s:M0n}.

The question of what can be deduced about an embedded variety from an
ambient toric variety has been a theme in the literature, with
variants of the cones $\G$, $\L$, and $\Uw$ appearing in special
cases.  This paper provides a unifying framework generalizing
these constructions.

In the context of mirror symmetry, 
Cox and Katz conjectured a description for the toric part of the nef cone of a
Calabi Yau hypersurface in a toric variety $X_{\Delta}$  \cite[Conjecture 6.2.8]{CoxKatz}.  We show that this is the cone $\G$; see
Lemma~\ref{l:CoxKatzisG}.  The subsequent
counterexamples to this conjecture and its variants
(\cite{SzendroiCoxKatz}, \cite{HassettLinWang}, \cite{SzendroiAmple},
\cite{Buckley}) give examples of the lower bound given by $\G$ not
being exact.

Another interesting class is given by Mori dream spaces~\cite{HuKeel},
important examples of which are log Fanos of general type~\cite{BCHM}.
A Mori dream space $Y$ has a natural embedding into a non-complete
toric variety $X_{\Delta}$ for which the induced map $i^* \colon
\Pic(X_{\Delta})_{\mathbb{R}} \rightarrow N^1(Y)_{\mathbb R}$ is an
isomorphism.  One may regard the ambient toric variety $X_{\Delta}$ as
a Rosetta Stone, encoding all birational models of $Y$; see
\cite{HuKeel}, \cite{BerchtoldHausen}, and \cite{Hausen2}.  In this
case the nef cone of $Y$ equals the pullback of $\G$ and $\L$; see
Section~\ref{s:MDS}.  The lower bounds $\G$ and $\L$ may thus be
considered as generalizations to arbitrary varieties of the
construction of the nef cone of a Mori dream space.  

The moduli space $\MOn$ of stable genus zero curves with $n$ marked
points also has an embedding into a non-complete toric variety
$X_{\Delta}$ with $\Pic(X_{\Delta})_{\mathbb{R}} \cong
N^1(\MOn)_{\mathbb R}$, where $\Delta$ is the space of phylogenetic
trees from biology (see \cite{Tevelev}, \cite{GMEquations}).  The nef
cone of $\MOn$ is famously unknown, with a possible description given
by the $\operatorname{F}$-Conjecture.  We show in
Proposition~\ref{p:GLF} that this cone equals the pullback of $\Uw$.
The original motivation for the $\operatorname{F}$-Conjecture came
from an analogy between $\MOn$ and toric varieties.  It suggests that,
as for complete toric varieties, the one-dimensional boundary strata
of $\overline{\operatorname{M}} _{0,n}$ should span its cone of
curves.  The interpretation of the conjecture as $\Nef(\MOn)=
\Uw(\MOn)$ thus deepens and explains this connection.

In addition, the upper bound $\mathcal{F}_{\Delta}(Y)$ can be considered to be a generalization of
the $\operatorname{F}$-Conjecture to  varieties $Y$ that can be realized as tropical 
compactifications.  The lower bound cone $\mathcal{L}_{\Delta}(Y)$ is defined for 
an even wider class of varieties,  and we propose that for $\MOn$ the lower bound cone $\L(\MOn)$ equals the
nef cone.  

Explicitly, let $\mathcal{I}=\{I \subset \{1,\ldots,n\}: 1
\in I \text{ and } |I|,|I^c| \geq 2\}$, and let $\delta_I$ denote the
boundary divisor on $\MOn$ corresponding to $I \in \mathcal{I}$.  We denote by $\pos(v_1,\dots,v_r)$ the positive hull $\{\sum_{i=1}^{r}\lambda_i v_i : \lambda_i \ge 0\}$ of a finite set of vectors $\{v_1,\dots,v_r\}$ in $\mathbb{R}^n$.
Then $\L(\MOn)$ equals
$$ \bigcap_{\sigma} \pos(\delta_I , \pm \delta_J : I, J \in
\mathcal{I} \setminus \sigma, \\ \delta_I \cap \delta_K \ne \emptyset,
\ \forall K \in \sigma, \text{and } \delta_J \cap \delta_L =
\emptyset \mbox{ for some } L \in \sigma),$$ where the intersection is
over all subsets $\sigma = \{ I_1,\dots, I_{n-3} \}$ of $n-3$ distinct
elements of $\mathcal I$ for which $\cap_{j=1}^{n-3} \delta_{I_j} \neq
\emptyset$.

For $n \leq 6$ we have verified that $\L(\MOn) = \nef(\MOn)=\Uw(\MOn)$.
While this may be true in general, as we are inclined to believe, we
feel that showing $\L(\MOn) = \nef(\MOn)$ may be more accessible than
the $\operatorname{F}$-conjecture.  The resulting description of
$\nef(\MOn)$ shares the main advantage of that given by the
$\operatorname{F}$-conjecture in that it provides a concrete
polyhedral description of the nef cone, allowing detailed analysis.

We now outline the structure of the paper.  The definitions of the
cones $\G$, $\L$, and $\Uw$ for a toric variety $X_{\Delta}$ are given
in Section~\ref{s:Cones}, along with several equivalent combinatorial interpretations.  In Section~\ref{bounds} we prove the main
result, Theorem~\ref{t:bound}.  Section~\ref{Examples} is then
devoted to applying Theorem~\ref{t:bound} to many classes of examples.
In Section~\ref{ss:delpezzo} we consider the resulting bounds on the
nef cone of a del Pezzo surface.  In Section~\ref{ss:CY} we consider
the case that $Y$ is an ample hypersurface in a toric variety and
connections with the Cox/Katz conjecture, while in
Section~\ref{ss:PicardRankTwo} we consider embeddings into toric
varieties of Picard rank two.  The application to Mori dream spaces is
described in Section~\ref{s:MDS}.  Finally, in Section~\ref{s:M0n} we
apply the main theorem to the moduli space $\MOn$.

{\bf Acknowledgements: } We
would like to thank Paul Hacking, Eric Katz, Sean Keel,
Danny Krashen, James McKernan, Sam Payne, Kevin Purbhoo, Frank Sottile, David Speyer,
Bernd Sturmfels and Jenia Tevelev for helpful and interesting conversations related to this work.  The authors were partially
supported by NSF grants DMS-0509319 (Gibney) and DMS-0500386
(Maclagan).

\section{Cones of divisors on $X_{\Delta}$}
\label{s:Cones}

Let $X_{\Delta}$ be a normal toric variety with fan $\Delta$.  
In this section we define three cones in $\Pic(X_{\Delta})_{\mathbb
R}$: $$\G \subseteq \L \subseteq \Uw.$$ 
When $X_{\Delta}$ is complete all cones are equal to
the nef cone of $X_{\Delta}$.

We mostly follow the notation for normal toric varieties of Fulton's
book~\cite{Fulton}, which we briefly recall.  Fix a lattice $N \cong
\mathbb Z^n$, and let $N_{\mathbb R} = N \otimes \mathbb R$.  Let $M =
\Hom(N,\mathbb Z)$.  We denote the pairing of $\mathbf{u} \in M$ and
$\mathbf{v} \in N$ by $\langle \mathbf{u} , \mathbf{v}\rangle$. 

Throughout this paper $\Delta$ will be a fan in $N_{\mathbb R}$ that
is not contained in any proper subspace.  Often we will assume that
$\Delta$ is pure of dimension $d$.  We emphasize that almost always we
will have $d<n$, so the corresponding toric variety, which we denote
by $X_{\Delta}$, is not complete.  We denote by $\Delta(k)$ the set of
cones of $\Delta$ of dimension $k$ for $0 \leq k \leq n$, and by
$|\Delta|$ the support $\{ v \in N_{\mathbb R} : v \in \sigma \text{
  for some } \sigma \in \Delta \}$.  We use the notation $i \in
\sigma$ for $\sigma \in \Delta$ to denote that the $i$th ray of
$\Delta$ is a ray of the cone $\sigma$.  By a piecewise linear
function $|\Delta| \rightarrow \mathbb R$ we will mean one that is
linear on each cone of $\Delta$.

For $i \in \Delta(1)$ we write $D_i$ for the corresponding torus-invariant divisor.  Any Weil divisor on $X_{\Delta}$ is linearly equivalent to one of the
form $D=\sum_{i \in \Delta(1)} a_i D_i$.  The divisor $D$ is
$\mathbb Q$-Cartier if there is a piecewise linear function $\psi
\colon |\Delta| \rightarrow \mathbb R$ with $\psi(\mathbf{v_i}) =
a_i$, where $\mathbf{v_i}$ is the first lattice point on the $i$th ray
of $\Delta$.  We write $\psi(\mathbf{v}) = - \langle \mathbf{u}(\sigma),
\mathbf{v} \rangle $ for $\mathbf{v} \in \sigma$ and $\mathbf{u}(\sigma) \in
M_{\mathbb R}$.  An element $\mathbf{u} \in M$ defines a global linear
function that corresponds to the Cartier divisor
$\divv(\chi^{\mathbf{u}}) = \sum \langle \mathbf{u},\mathbf{v}_i \rangle D_i$.
We can thus identify $\Pic(X_{\Delta})_{\mathbb R}$ with the set of
piecewise linear functions on $|\Delta|$ modulo global linear
functions.

We denote by $\pos( S )$ the cone in $\Pic(X_{\Delta})_{\mathbb R}$
generated by a collection of divisors $S \subseteq
\Pic(X_{\Delta})_{\mathbb R}$, where $S$ is some set.  Our first cone
is the following:

\begin{definition} 
The cone $\G$ is the set 
$$\G = \pos( [D] \in \Pic(X_{\Delta})
: D \text{ is globally generated} ).$$
\end{definition}

The cone $\G$ can be computed directly from the fan $\Delta$, as
Proposition~\ref{p:Gcone} below illustrates.  For the strongest result
we will need the following additional hypothesis.

\begin{hypothesis} \label{hyp:quasiproj}
There is a projective toric variety $X_{\Sigma}$ 
with 
$\Delta \subseteq \Sigma$.
\end{hypothesis}

A function $\psi : N_{\mathbb R} \rightarrow \mathbb R$ is
{\em convex} if $\psi(\sum_{i=1}^l \mathbf{u}_i) \leq \sum_{i=1}^l
\psi(\mathbf{u}_i)$ for all choices of $\mathbf{u}_1,\dots,
\mathbf{u}_l \in N_{\mathbb R}$.

\begin{proposition} \label{p:Gcone}

Fix $[D] \in \Pic(X_{\Delta})$ with $D=\sum_{i \in \Delta(1)} a_i D_i$.  Then
the following are equivalent:  \newcounter{saveenum}
\begin{enumerate}
\item \label{i:Gdefn} $[D] \in \G$;

\item \label{i:Gasintersection} $[D] \in \bigcap_{\sigma \in \Delta}
\pos([D_{i}] : i \not \in \sigma)$;

\item \label{i:Gconepsi} 
There is a piecewise linear convex function $\psi \colon N_{\mathbb R}
\rightarrow \mathbb R$ that is linear on the cones of $\Delta$ with
$\psi(\mathbf{v}_i) = a_i$.
 \setcounter{saveenum}{\value{enumi}}
\end{enumerate}

If in addition $\Delta$ satisfies Hypothesis~\ref{hyp:quasiproj} then these are also equivalent to:
\begin{enumerate}
\setcounter{enumi}{\value{saveenum}}
\item \label{i:Gunionnef} $[D] \in \bigcup_{\Sigma}
i_{\Sigma}^*(\nef(X_{\Sigma}))$, where the union is over all
projective toric varieties $X_{\Sigma}$ with $\Delta \subset \Sigma$
and $i_{\Sigma}$ is the inclusion morphism of $X_{\Delta}$ into
$X_{\Sigma}$.  This union is equal to the union restricted to those
$\Sigma$ with $\Sigma(1)=\Delta(1)$.

\end{enumerate}

\end{proposition}

\begin{proof}
{\bf \ref{i:Gdefn} $\leftrightarrow$ \ref{i:Gasintersection}}:
Recall that $D$ is globally generated if and only if for each $\sigma
\in \Delta$ there is a $\mathbf{u}(\sigma) \in M$ for which
$\langle \mathbf{u}(\sigma), \mathbf{v}_i \rangle \geq -a_i$ for all $i$, and
$\langle \mathbf{u}(\sigma) , \mathbf{v}_i \rangle  = -a_i$ when $i \in \sigma$
(see \cite[p68]{Fulton}).  If $[D] \in \pos([D_{i}] : i \not \in
\sigma)$ then there a representative of $[D]$ of the form $\sum_{i
\not \in \sigma} a_i D_i$, where $a_i \geq 0$ for all $i$, so we can
take $\mathbf{u}(\sigma)=0$ for this $D$.  Conversely, if $D$ is
globally generated, then for each $\sigma$ we note that $D +
\mathrm{div}(\chi^{\mathbf{u}(\sigma)})$ is an effective combination of $\{
D_i : i \not \in \sigma \}$, so $[D] \in \pos([D_i] : i \not \in
\sigma)$ for each $\sigma \in \Sigma$.

{\bf \ref{i:Gdefn} $\rightarrow$ \ref{i:Gconepsi}}: Let $P_D$ be the
polyhedron $\{ \mathbf{u} \in M_{\mathbb R} : \langle \mathbf{u},
\mathbf{v}_i \rangle \geq -a_i \text{ for all } i \}$.  Then $D$ is globally
generated if and only if each cone $\sigma$ of $\Delta$ is contained
in a cone of the inner normal fan $\mathcal N$ of $P_D$.  Let $\psi
\colon N_{\mathbb R} \rightarrow \mathbb R$ be defined by
$\psi(\mathbf{v}) = -\min_{\mathbf{u} \in P_D} \langle \mathbf{u},
\mathbf{v} \rangle$.  Then $\psi$ is a convex function that is linear
on the cones of $\mathcal N$,  and thus linear on the cones of
$\Delta$, with $\psi(\mathbf{v}_i)=\mathbf{a}_i$, as required.

{\bf \ref{i:Gconepsi} $\rightarrow$ \ref{i:Gasintersection}}: Suppose
that there is a piecewise linear convex function $\psi$ on $N_{\mathbb
  R}$ that is linear on the cones of $\Delta$ with
$\psi(\mathbf{v}_i)=a_i$.  Then for any fixed $\sigma \in \Delta$
there is a full-dimensional cone $\tau \subseteq N_{\mathbb R}$
containing $\sigma$ and $\mathbf{u} \in M_{\mathbb R}$ for which
$\psi(\mathbf{v}) = - \langle \mathbf{u} , \mathbf{v} \rangle$ for all
$\mathbf{v} \in \tau$.  Fix $\mathbf{v} \in \intt(\tau)$.  Since
$\tau$ is full-dimensional, for any $i$ there is $0<\lambda<1$ with
$\mathbf{v}'=(1-\lambda)\mathbf{v} +\lambda \mathbf{v}_i \in \tau$.
Since $\psi$ is convex, we have $\psi(\mathbf{v}') = -\langle
\mathbf{u}, \mathbf{v}' \rangle = -(1-\lambda) \langle \mathbf{u},
\mathbf{v} \rangle - \lambda \langle \mathbf{u}, \mathbf{v}_i \rangle
\leq \psi((1-\lambda)\mathbf{v}) + \psi( \lambda \mathbf{v}_i) =
-(1-\lambda) \langle \mathbf{u}, \mathbf{v} \rangle + \lambda a_i$, so
$a_i + \langle \mathbf{u}, \mathbf{v}_i \rangle \geq 0$.  If
$\mathbf{v}_i \in \tau$ we have $a_i + \langle \mathbf{u},
\mathbf{v}_i \rangle = 0$.  This implies that
$D'=D+\divv(\chi^{\mathbf{u}}) \in \pos([D_i] : i \not \in \sigma)$,
and thus $[D] \in \cap_{\sigma \in \Delta} \pos([D_i] : i \not \in
\sigma)$.

{\bf \ref{i:Gunionnef} $\rightarrow$ \ref{i:Gdefn}}: Let
$i \colon X_{\Delta} \rightarrow X_{\Sigma}$ be an inclusion of
$X_{\Delta}$ into a projective toric variety $X_{\Sigma}$.  If $[D]=
i^*([D'])$ for some nef, and thus globally generated, divisor class
$[D']$ on $X_{\Sigma}$, then $[D]$ is globally generated,
since the pullback of a globally generated divisor is globally
generated.
 Thus $\bigcup_{\Sigma} i^*(\nef(X_{\Sigma})) \subseteq \mathcal
G(X_{\Delta})$.

{\bf \ref{i:Gasintersection} $\rightarrow$ \ref{i:Gunionnef}}: We
first note that Hypothesis~\ref{hyp:quasiproj} implies that the
intersection over all $\sigma \in \Delta$ of the relative interiors of
$\pos(D_{i} : i \not \in \sigma)$ is nonempty.  To see this consider a
projective toric variety $X_{\Sigma}$ with $\Delta \subset \Sigma$
whose existence is guaranteed by Hypothesis~\ref{hyp:quasiproj}.  Let
$D'_i$ denote the torus invariant divisor corresponding to the $i$th ray
of $\Sigma$.  For each $\sigma \in \Delta$ fix a Cartier divisor
$D'_{\sigma} = \sum_{i \in \Sigma(1)} (a_{\sigma})_i D'_i$ with
$(a_{\sigma})_i>0$ for $i \not \in \sigma$ and $(a_{\sigma})_i=0$ for
$i \in \sigma$.  Let $D'$ be an ample Cartier divisor on
$X_{\Sigma}$.  We may choose $D'$ sufficiently positive so that
$D'-D'_{\sigma}= \sum (a'_{\sigma})_i D'_i$ is also ample for all
$\sigma \in \Delta$.  Then for any $\sigma \in \Delta$, since $\sigma
\in \Sigma$ there is $\mathbf{u}(\sigma) \in M$ with $\langle \mathbf{u}(\sigma),
\mathbf{v}_i \rangle = -(a'_{\sigma})_i $ for $i \in \sigma$, and
$\langle \mathbf{u}(\sigma), \mathbf{v}_i \rangle \geq -(a'_{\sigma})_i $ for
$i \not \in \sigma$.  Then
$[D']=[D'-D'_{\sigma}+\divv(\chi^{\mathbf{u}(\sigma)})] + [D'_{\sigma}]
=\sum_{i \not \in \sigma} (a''_{\sigma})_i D'_i$, where
$(a''_{\sigma})_i > 0$ for all $i \not \in \sigma$.  Thus $i^*([D']) \in
\bigcap_{\sigma \in \Delta} \relint(\pos(D_i : i \not \in \sigma))$.

Since this intersection is nonempty, its closure is $\bigcap_{\sigma
  \in \Delta} \pos(D_{i} : i \not \in \sigma)$.  Thus
to show that $\bigcap_{\sigma \in \Delta} \pos(D_i : i \not \in
\sigma) \subseteq \bigcup_{\Sigma } i^*(\nef(X_{\Sigma}))$ it suffices
to show that $\bigcap_{\sigma \in \Delta} \relint(\pos(D_{i} : i \in
\sigma)) \subseteq \bigcup_{\Sigma : \Sigma(1)=\Delta(1)}
i^*(\nef(X_{\Sigma}))$, as this latter, a priori smaller, set is a
finite union of closed sets.  Let $D = \sum a_i D_i \in
\bigcap_{\sigma \in \Delta} \relint(\pos(D_{i} : i \not \in \sigma))$.
Then the regular subdivision $\Sigma'$ of $\{ \mathbf{v}_i \}$ induced
by the $a_i$ contains $\sigma$ as a face for all $\sigma \in \Delta$
(see, for example, \cite[Chapter 8]{GBCP}).  Let $D'=\sum a_i D'_i$
denote the corresponding divisor on $X_{\Sigma'}$.  By construction
$P_{D'}$ has $\Sigma'$ as its normal fan, so $D'$ is ample; see
\cite[pp66-70]{Fulton}.  Thus $[D]=i^*([D']) \in
i^*(\nef(X_{\Sigma'}))$.  Note that $\Sigma'(1)=\Delta(1)$ by
construction, so this shows that $\G \subseteq \bigcup_{\Delta \subset
  \Sigma, \Delta(1)=\Sigma(1)} i^*(\nef(X_{\Sigma})) \subseteq
\bigcup_{\Delta \subset \Sigma} i^*(\nef(X_{\Sigma}))$.  Since we have
already shown the inclusion $\bigcup_{\Delta \subset \Sigma}
i^*(\nef(X_{\Sigma})) \subseteq \G$, we conclude that all three sets
coincide.
\end{proof}

\begin{remark}
 To see that Hypothesis~\ref{hyp:quasiproj} is needed for the
equivalence of the last item, consider the complete three-dimensional
toric variety $X_{\Delta}$ whose fan intersects the sphere
$S^2$ as shown in Figure~\ref{f:nonregular}, and for which $\mathbf{v}_1,\dots,\mathbf{v}_7$ are the columns of the matrix $V$ below.
In the picture, vertex $7$ has been placed at infinity.
\begin{figure}
\center{\resizebox{5cm}{!}{\input{nonregular.pdftex_t}}}
\caption{\label{f:nonregular}} 
\end{figure}

Then $\Pic(X_{\Delta})_{\mathbb R} \cong \mathbb Z^{4}$, with an isomorphism taking $[D_i]$ to the $i$th column of the matrix $G$:
\renewcommand{\arraystretch}{0.8}
\renewcommand{\arraycolsep}{2pt}
$$V= \left( \text{\footnotesize $\begin{array}{rrrrrrr}
3 & 0 & 0 & 2 & 1 & 1 & -1 \\
0 & 3 & 0 & 1 & 2 & 1 & -1 \\
0 & 0 & 3 & 1 & 1 & 2 & -1 \\
\end{array}$}
\right), \, \, \, \,
G = \left( \text{\footnotesize $ \begin{array}{rrrrrrr}
-1 & -1 & -1 & 1 & 1 & 1 & 1 \\
-1 & -2 & -1 & 0 & 3 & 0 & 0\\
-1 & -1 & -2 &0 & 0 & 3 & 0\\
1 & 1 & 1 & 0 & 0 & 0 & 3 \\
\end{array}$}
\right).$$ The cone $\G$ equals  $\pos((1,0,0,3))$.  This can be obtained from the description of
Part~\ref{i:Gasintersection} of Proposition~\ref{p:Gcone} using
software such as PORTA~\cite{Porta}.  There is no projective
toric variety $X_{\Sigma}$ with 
$\Delta \subseteq \Sigma$, as we would have to have $\Sigma=\Delta$
since $\Delta$ is complete, and $X_{\Delta}$ is not
projective.

In such cases we may recognize $\G$ as the union of the pullbacks of
the nef cones of all {\em complete} toric varieties $X_{\Sigma}$ with
$\Delta \subseteq \Sigma$.  Indeed, the proof of (\ref{i:Gunionnef}
$\rightarrow$ \ref{i:Gdefn}) of Proposition~\ref{p:Gcone} goes through
unchanged.  The other inclusion can be obtained by modifying the
second paragraph of (\ref{i:Gasintersection} $\rightarrow$
\ref{i:Gunionnef}) (ignoring the first paragraph, which no longer
applies), by taking a a refinement of the fan $\Sigma'$ so that
$\Delta$ is still a subfan.  The more restrictive statement of
Proposition~\ref{p:Gcone}(\ref{i:Gunionnef}), however, is more relevant
in our applications in the rest of this paper.
\end{remark}

A cone $\sigma \in \Delta$ determines a torus orbit $\mathcal O(\sigma)$ which
is isomorphic to $(\K^*)^{n-\dim(\sigma)}$.  We denote by $V(\sigma)$ the closure of
$\mathcal O({\sigma})$ in $X_{\Delta}$.

\begin{definition} \label{d:Lcone}
For $\sigma \in \Delta$ we denote by $i_{\sigma} : V(\sigma)
\rightarrow X_{\Delta}$ the inclusion of the orbit closure
$V(\sigma)$ into $X_{\Delta}$.
 Then
$$\L = \{ [D] \in \Pic(X_{\Delta})_{\mathbb R}
: i_{\sigma}^{*}([D]) \text{ is effective for all } \sigma \in
\Delta \}.$$ 
\end{definition}

\begin{remark}
Note that by taking $\sigma=\{0\}$ we see that $\L$ is contained in
the effective cone of $X_{\Delta}$.  Note also that when $\sigma$ is a
maximal cone of $\Delta$,  $V(\sigma) \cong (\K^*)^{n-\dim(\sigma)}$, so
the condition that $i_{\sigma}^*([D])$ is effective is vacuous.  Let
$\Delta^{\circ}$ be the subfan of $\Delta$ with all maximal cones removed.
We thus have $\L = \{ [D] \in \Pic(X_{\Delta} )_{\mathbb R} : i_{\sigma}^*([D])
\text{ is effective for all } \sigma \in \Delta^{\circ} \}$.
\end{remark}

Recall that the {\em star} of a cone $\sigma \in \Delta$ is the fan
$\starr(\sigma)$ whose cones are $\{ \tau \in \Delta : \sigma
\subseteq \tau \}$ together with all faces of these cones.  Let
$N_{\sigma}$ be the lattice generated by $N \cap \sigma$.  For $\tau
\in \starr(\sigma)$ we denote by $\overline{\tau}$ the cone $(\tau+
N_{\sigma} \otimes \mathbb R )/N_{\sigma} \otimes \mathbb R$.  The
cones $\{\overline{\tau} : \tau \in \starr(\sigma) \}$ form a fan in
$(N/N_{\sigma}) \otimes \mathbb R$, and the corresponding toric
variety is $V(\sigma)$.
The fan $\starr^1(\sigma)$ is the subfan of $\starr(\sigma)$ whose
top-dimensional cones are $\{ \tau \in \starr(\sigma) : \dim(\tau) =
\dim(\sigma)+1 \}$.  The top-dimensional cones of $\starr^1(\sigma)$
correspond to rays of the fan of $V(\sigma)$.  We denote by
$\mathbf{e}_{\tau}$ the first lattice point on the ray corresponding
to $\tau \in \starr^1(\sigma)$. For any $i \in \tau \setminus \sigma$
we have $\overline{\mathbf{v}}_i = c_i \mathbf{e}_{\tau}$ for some
integer $c_i > 0$, where $\overline{\mathbf{v}}_i$ denotes the image
of $\mathbf{v}_i$ in $(N/N_{\sigma}) \otimes \mathbb R$.

Fix $\sigma \in \Delta$, a Cartier divisor $D = \sum a_i D_i$, and
choose $u(\sigma) \in M$ satisfying $\langle u(\sigma), \mathbf{v}_i
\rangle = -a_i$.  We will use the following formula for
$i_{\sigma}^*([D])$:
\begin{equation} \label{e:pullback}
i_{\sigma}^*(D)
= \sum_{\tau \in \starr^1{\sigma}} b_{\tau} [D_{\tau}],
\end{equation}
 where $b_{\tau} = (a_i+\langle \mathbf{u}(\sigma), \mathbf{v}_i
 \rangle )/c_i$ for any $i \in \tau \setminus \sigma$. This is
 independent of the choice of $i$; see \cite[p. 97]{Fulton}.

 A function $\psi \colon |\Delta| \rightarrow \mathbb
R$ linear on the cones of $\Delta$ is {\em convex on
$\starr^1(\sigma)$} if the inequality  $\psi(\sum_{i=1}^l \mathbf{u}_i) \leq
\sum_{i=1}^l \psi(\mathbf{u}_i)$ holds for all
$\mathbf{u}_1,\dots,\mathbf{u}_l \in \starr^1(\sigma)$ with
$\sum_{i=1}^l \mathbf{u}_i \in \sigma$.

\begin{proposition} \label{p:pisigmastar}
Fix $[D] \in \Pic(X_{\Delta})$ with $D=\sum_{i \in \Delta(1)}  a_i D_i$,
and $\sigma \in \Delta$.  Let $i_{\sigma} \colon V(\sigma)
\rightarrow X_{\Delta}$ be the inclusion map.
Then the following are equivalent:

\begin{enumerate}
\item \label{i:sigmaeffective}$i_{\sigma}^*([D])$ is effective;

\item \label{i:sigmapos} $[D]=[D']$, where $D'=\sum a'_i D_i$ with $a'_i=0$ for $i \in
\sigma$ and $a'_i \geq 0$ for $i \in \starr^1(\sigma)$;

\item \label{i:piecewisesigmacone} The piecewise linear function $\psi_D
\colon |\Delta| \rightarrow \mathbb R$ defined by setting
$\psi_D(\mathbf{v}_i)=a_i$ and extending to be linear on each cone
$\tau \in \Delta$ is convex on $\starr^1(\sigma)$;

\item \label{i:sigmainequal} $\sum_i a_i b_i \geq 0$ for all $\mathbf{b}=(b_i)$ with
$\sum_{i \in \Delta(1)}b_i \mathbf{v}_i =0$ such that $b_i = 0$
for $i \not \in \starr^1(\sigma)$, and $b_i \geq 0$ for $i \in
\starr^1(\sigma) \setminus \sigma$.

\setcounter{saveenum}{\value{enumi}}
\end{enumerate}

\end{proposition}

\begin{proof}
{\bf \ref{i:sigmaeffective} $\rightarrow$ \ref{i:sigmapos} }:  
Suppose $i_{\sigma}^*([D])$ is effective.  We may assume that the
representative for $D$ has been chosen so that $a_i = 0$ for $i \in
\sigma$ (by replacing $D$ by $D+\divv(\chi^{\mathbf{u}(\sigma)})$), so
by Equation~\ref{e:pullback} we have $i_{\sigma}^*([D]) = \sum_{\tau \in
  \starr^1(\sigma)} b_{\tau} D_{\tau}$, where $b_{\tau} = a_i/c_i$ for
any $i \in \tau \setminus \sigma$, and $\overline{\mathbf{v}}_i = c_i
\mathbf{e}_{\tau} \in (N/N_{\sigma})\otimes \mathbb{R}$.
Since $i_{\sigma}^{*}([D])$ is effective there is a representative
$i_{\sigma}^*([D]) = \sum b'_{\tau} D_{\tau}$ with $b'_{\tau} \geq 0$ for all
$\tau \in \starr^1(\sigma)$.  This means that there is $\mathbf{u} \in
\Hom(N/N_{\sigma},\mathbb{R})$ with $b'_{\tau} =
b_{\tau} + \langle \mathbf{u}, \mathbf{e}_{\tau} \rangle$.  Let
$\tilde{\mathbf{u}}$ be the lift of $\mathbf{u}$ to $M_{\mathbb R}$
with $\langle \tilde{\mathbf{u}},\mathbf{v}_i \rangle = 0$ for $i \in
\sigma$.  Let $D' = D + \divv(\chi^{\tilde{\mathbf{u}}})$.  Then $D' =
\sum a'_i D_i$, where by construction $a'_i=0$ for $i \in \sigma$.  In
addition, for $i \in \tau \setminus \sigma$ with $\tau \in
\starr^1(\sigma)$ we have $a'_i = a_i +c_i \langle \mathbf{u},
\mathbf{e}_{\tau} \rangle = c_i b'_{\tau} \geq 0$, so $D'$ has the
desired form.

{\bf \ref{i:sigmapos} $\rightarrow$ \ref{i:piecewisesigmacone}}:
Note first that if $D' = D + \divv(\chi^{\mathbf{u}})$ for some
$\mathbf{u} \in M$ then $\psi_{D'}(\mathbf{v}) = \psi_D(\mathbf{v}) +
\langle \mathbf{u}, \mathbf{v} \rangle$.  Thus $\psi_D$ is convex on
$\starr^1(\sigma)$ if and only if $\psi_{D'}$ is.  Suppose now that
$[D]=[D']$ for $D'= \sum a'_i D_i$ with $a_i = 0$ for $i \in \sigma$
and $a_i \geq 0$ for $i \in \starr^1(\sigma)$.  Let
$\mathbf{u}_1,\dots, \mathbf{u}_l \in \starr^1(\sigma)$ with
$\sum_{i=1}^l \mathbf{u}_i \in \sigma$.  Then $\psi_{D'}(\mathbf{u}_i)
\geq 0$ for all $i$, and $\psi_{D'}(\sum \mathbf{u}_i) = 0$, so
$\psi_{D'}$ is convex on $\starr^1(\sigma)$.  

{\bf \ref{i:piecewisesigmacone} $\rightarrow$ \ref{i:sigmainequal}}:
Suppose that $\psi_D$ is convex on $\starr^1(\sigma)$.  By
replacing $D$ by $D + \divv(\chi^{\mathbf{u}(\sigma)})$ we may assume
that $\psi_D(\mathbf{v}) = 0$ for all $\mathbf{v} \in \sigma$.  Let
$\mathbf{b} \in \mathbb R^{|\Delta(1)|}$ satisfy $\sum b_i
\mathbf{v}_i = 0$, $b_i =0 $ for $i \not \in \starr^1(\sigma)$, and
$b_i \geq 0$ for $i \in \starr^1(\sigma) \setminus \sigma$.  Let
$\mathbf{u}_i=b_i \mathbf{v}_i$ for $i \in \starr^1(\sigma) \setminus
\sigma$, and $\mathbf{u}_0=\sum_{j \in \sigma, b_j >0} b_j
\mathbf{v}_j$.  Then $\mathbf{u}_i \in \starr^1(\sigma)$ for all $i$,
and $\mathbf{u}_0+ \sum_{i \in \starr^1(\sigma) \setminus \sigma}
\mathbf{u}_i = \sum_{j \in \sigma, b_j <0} (-b_j) \mathbf{v}_j \in
\sigma$, so since $\psi_D$ is convex on $\starr^1(\sigma)$, we have
$\sum \psi_D(\mathbf{u}_i) \geq \psi_D(\sum_{j \in \sigma, b_j <0}
(-b_j) \mathbf{v}_j ) = 0$.  Now $\psi_D(\mathbf{u}_i) = b_i a_i$ for
$i \in \starr^1(\sigma) \setminus \sigma$, and $\psi_D(\mathbf{u}_0)
=0$, so $\sum_{i \in \Delta(1)} a_i b_i = \sum_{i \in \starr^1(\sigma) \setminus \sigma}
a_i b_i \geq 0$ as required.

{\bf{ \ref{i:sigmainequal} $\rightarrow$ \ref{i:sigmaeffective}}}:
Suppose that $\sum_i a_ib_i \geq 0$ for all $\mathbf{b} \in \mathbb
R^{|\Delta(1)|}$ with $\sum b_i \mathbf{v}_i =0$, $b_i \geq 0$ for $i
\not \in \sigma$ and $b_i = 0 $ for $ i \not \in \starr^1(\sigma)$.
Let $i_{\sigma}^*(D) = \sum d_{\tau} D_{\tau}$.  To show that $i_{\sigma}^*(D)$ is
effective, it suffices to show that it lies on the correct side of all
facet-defining hyperplanes of the effective cone, and thus that
$\sum_{\tau} \tilde{b}_{\tau} d_{\tau} \geq 0$ for all choices of
$\tilde{b}_{\tau} \geq 0$ with $\sum_{\tau \in \starr^1(\sigma)}
\tilde{b}_{\tau} \mathbf{e}_{\tau} =0$.  Given such a vector
$\tilde{b}$, we construct $\mathbf{b} \in \mathbb R^{|\Delta(1)|}$
with $\sum b_i \mathbf{v}_i =0$ as follows.  For each $\tau \in
\starr^1(\sigma)$ choose $i \in \tau \setminus \sigma$, and set
$b_i=\tilde{b}_{\tau}/c_i$, where as above $\overline{\mathbf{v}}_i =
c_i \mathbf{e}_{\tau}$.  Set $b_j=0$ for all other $j \in \tau
\setminus \sigma$, and for $j \not \in \starr(\sigma)$.  We then have
$\sum_{\tau \in \starr^1(\sigma)} \sum_{i \in \tau \setminus \sigma}
b_i \mathbf{v}_i \in N_{\sigma}$.  Choose $b_j \in \mathbb Z$ so that
this sum is $\sum_{j \in \sigma} -b_j \mathbf{v}_j$.  Then by
construction $\sum b_i \mathbf{v}_i =0$, $b_i \geq 0$ for $i \not \in
\sigma$, and $b_i=0$ for $i \not \in \starr(\sigma)$.  Thus $\sum a_i
b_i \geq 0$.  Now $d_{\tau} = (a_i+\langle \mathbf{u}(\sigma),
\mathbf{v}_i\rangle)/c_i$ so $\sum_{\tau} \tilde{b}_{\tau} d_{\tau} =
\sum_{\tau} \sum_{i \in \tau \setminus \sigma} b_ic_i(a_i + \langle
\mathbf{u}(\sigma), \mathbf{v}_i \rangle )/c_i = \sum_{i \in
  \starr^1(\sigma)\setminus \sigma} b_ia_i + \sum_{i \in
  \starr^1(\sigma)\setminus \sigma} b_i \langle \mathbf{u}(\sigma),
\mathbf{v}_i \rangle$.  Since $\sum b_i \mathbf{v}_i=0$, we have
$\sum_{i \in \starr^1(\sigma)\setminus \sigma)} b_i \langle
\mathbf{u}(\sigma), \mathbf{v}_i \rangle = \sum_{j \in \sigma} b_j
a_j$.  Thus $\sum_{\tau} \tilde{b}_{\tau} d_{\tau} = \sum a_i b_i \geq
0$, so $i_{\sigma}^*(D)$ is effective.

\end{proof}

\begin{corollary} \label{c:Lcone}
Fix $[D] \in \Pic(X_{\Delta})$ with $D=\sum_{i \in \Delta(1)} a_i D_i$.
Then the following are equivalent:

\begin{enumerate}
\item \label{i:defnLcone} $[D] \in \L$;

\item \label{i:modifiedGcone} $[D] \in \bigcap_{\sigma \in \Delta^{\circ}}
\pos( [D_i], \pm [D_j] : i \in \starr^1(\sigma) \setminus \sigma, j
\in \Delta(1) \setminus \starr^1(\sigma))$;

\item \label{i:piecewiseLcone} The piecewise linear function $\psi
\colon |\Delta| \rightarrow \mathbb R$ defined by setting
$\psi(\mathbf{v}_i)=a_i$ and extending to be linear on each cone
$\sigma \in \Delta$ is convex on $\starr^1(\sigma)$ for all $\sigma
\in \Delta^{\circ}$;

\item \label{i:inequalLcone} $\sum_i a_i b_i \geq 0$ for all
  $\mathbf{b}=(b_i)$ with $\sum_i b_i \mathbf{v}_i =0$ such that there
  is $\sigma \in \Delta$ with $b_i = 0$ for $i \not \in
  \starr^1(\sigma)$, and $b_i \geq 0$ for $i \in
  \starr^1(\sigma)\setminus \sigma$.

\setcounter{saveenum}{\value{enumi}}
\end{enumerate}

\end{corollary}

\begin{proof}
This follows directly from Proposition~\ref{p:pisigmastar}, since the set 
$\{[D] : [D]=[\sum a_i D_i] \\ \text{ with } a_i = 0 \text{ for } i \in
\sigma, a_i \geq 0 \text{ for } i \in \starr^1(\sigma) \}$ equals 
$\pos([D_i], \pm [D_j] : i \in \starr^1(\sigma) \setminus \sigma, j \in
\Delta(1) \setminus \starr^1(\sigma))$.

\end{proof}

\begin{remark}
Note that the convex function $\psi$ in Part~\ref{i:Gconepsi} of
Proposition~\ref{p:Gcone} is defined on all of $N_{\mathbb R}$, while
the function $\psi$ defined in Part~\ref{i:piecewiseLcone} of
Corollary~\ref{c:Lcone} is only defined on $|\Delta|$.  Also, the
first $\psi$ is required to be globally convex, while the second is
only locally convex (convex on $\starr^1(\sigma)$).

\begin{figure}
\center{\resizebox{8cm}{!}{\input{Hirzebruch.pdftex_t}}}
\caption{The cones $\G$ and $\L$ for the punctured surface
  $\mathbb{F}_1$. \label{f:Hirzebruch}}
\end{figure} 

To see this second difference, let $X_{\Delta}$ be the Hirzebruch
surface $\mathbb F_1$ with the four torus invariant points removed.
Name the torus invariant prime divisors $D_1,\dots,D_4$ as in
Figure~\ref{f:Hirzebruch}.  Then the Picard group is generated by
$[D_1]=[D_3]$ and $[D_4]$.  The cone $\mathcal{G}_{\Delta}$ is
$\bigcap_{i=1}^4 \pos([D_j]: j \neq i ) = \pos([D_1], [D_4])$.  The
cone $\L$ is equal to the effective cone $\pos([D_1],[D_2])$ of
$X_{\Delta}$, so $\G \subsetneq \L$.  This is illustrated in
Figure~\ref{f:Hirzebruch}.
\end{remark}

For the last cone we assume that every maximal cone in $\Delta$ has
dimension $d$.

\begin{definition}  \label{d:Fdef}
Fix $\mathbf{w} \in \Hom(A_{n-d}(X_{\Delta}), \mathbb R)$ and let
\begin{align*}
\Uw & = \{ D \in
\Pic(X_{\Delta})_{\mathbb{R}} : \mathbf{w}( [D] \cdot [V(\tau)] ) \geq 0 \ \ 
\text{ for all } \tau \in \Delta(d-1) \} \\
 & = \{ D\in \Pic(X_{\Delta})_{\mathbb R} : \sum
a^{\tau}_{\sigma} w_{\sigma} \geq 0 \text{ for all } \tau \in
\Delta(d-1)\}, 
\end{align*}
where $[D] \cdot [V(\tau)] = \sum_{\sigma \in \Delta(d)}
a^{\tau}_{\sigma} [V(\sigma)]$, and $w_{\sigma} =
\mathbf{w}(V(\sigma))$.
\end{definition}

\begin{definition} \label{d:Wdefn}
Let $W = \{ \mathbf{w} \in \Hom(A_{n-d}(X_{\Delta}), \mathbb R) : \mathbf{w}(V(\sigma)) \geq 0 \text{ for all } \sigma \in \Delta(d) \}$.  Let
$$\U = \bigcap_{\mathbf{w} \in W} \Uw.$$
\end{definition}

When $A_{n-d}(X_{\Delta}) \cong \mathbb Z$, then $\U = \Uw$ for all
$\mathbf{w} \in W$.  This is the case for many specific $\Delta$ of
interest; see Propositions~\ref{p:delpezzo} and \ref{chowisomorphism}.
We will use the following equivalent descriptions of $\U$.

\begin{proposition}  \label{p:Ucone}
Fix $[D] \in \Pic(X_{\Delta})$ with $D=\sum_{i \in \Delta(1)} a_i D_i$.
Then the following are equivalent:

\begin{enumerate}
\item \label{i:defnUcone} $[D] \in \U$;

\item \label{i:Upullbackeffective} $i_{\tau}^*(D)$ is effective for
the inclusions $i_{\tau} : V(\tau) \rightarrow X_{\Delta}$ with
$\tau \in \Delta(d-1)$;

\item \label{i:modifiedGconeforU} $[D] \in \bigcap_{\sigma \in
\Delta(d-1)} \pos( [D_i], \pm [D_j] : i \in \starr^1(\sigma) \setminus
\sigma, j \in \Delta(1) \setminus \starr^1(\sigma))$;

\item \label{i:inequalUcone} $\sum_i a_i b_i \geq 0$ for all
$\mathbf{b}=(b_i)$ with $\sum b_i \mathbf{v}_i =0$ such that there is $\sigma
\in \Delta(d-1)$ with $b_i = 0$ for $i \not \in \starr^1(\sigma)$, and $b_i
\geq 0$ for $i \in \starr^1(\sigma)\setminus \sigma$;

\item \label{i:piecewiseUcone} The piecewise linear function $\psi
\colon |\Delta| \rightarrow \mathbb R$ defined by setting
$\psi(\mathbf{v}_i)=a_i$ and extending to be linear on each cone
$\sigma \in \Delta$ is convex on $\starr^1(\sigma)$ for all $\sigma
\in \Delta(d-1)$.
\setcounter{saveenum}{\value{enumi}}
 \end{enumerate}
\end{proposition}

\begin{proof}
The equivalence of Part~\ref{i:Upullbackeffective} with the following
ones is a direct corollary of Proposition~\ref{p:pisigmastar}, as in
Corollary~\ref{c:Lcone}, so we thus need only show the equivalence of
the first condition with the others.  

Suppose first that $i_{\tau}^*(D)$ is effective for all $\tau \in
\Delta(d-1)$.  Then for all such $\tau$, the class $[D] \cdot
[V(\tau)]$ lies in the cone generated by $\{ [V(\sigma)] : \sigma \in
\starr(\tau), \dim(\sigma)=d \}$.  Thus $\mathbf{w}([D] \cdot
[V(\tau)]) \geq 0$ for all $\mathbf{w} \in W$, so $D \in \U$.  

Conversely, suppose that $D \in \U$.  Then for all $\tau \in
\Delta(d-1)$ the class $[D] \cdot [V(\tau)]$ must have a representative
$\sum_{\sigma \in \Delta(d)} a_{\sigma} [V(\sigma)]$, where
$a_{\sigma} \geq 0$, as otherwise there would be $\mathbf{w} \in W$
with $\mathbf{w}([D] \cdot [V(\tau)] )<0$.  Since $i_{\tau}^*([D]) =
\sum_{\sigma \in \starr^1(\tau)} a_{\sigma} D_{\sigma}$, this implies
that $i_{\tau}^*([D])$ is effective.
\end{proof}

Recall that $\Delta^{\circ}$ is the fan obtained from $\Delta$ by removing the maximal cones.

\begin{proposition} \label{p:toriccontainment}
We have 
$$\G \subseteq \Gprime \subseteq \L,$$ 
and if $\Delta$ is a pure fan of dimension $d$ then 
$$\L \subseteq \U \subseteq \Uw.$$ If $\Delta$ is the
fan of a projective toric variety $X_{\Delta}$, then all cones
coincide, and are equal to $\nef(X_{\Delta})$.
\end{proposition}

\begin{proof}
By Part~\ref{i:Gasintersection} of Proposition~\ref{p:Gcone} $\G$ is
the intersection of $\Gprime$ with some other cones, which gives the
first inclusion. Next, note that for $\sigma \in \Delta^{\circ}$, we have
$\pos([D_i] : i \not \in \sigma) \subseteq \pos( [D_i], \pm [D_j] : i
\in \starr^1(\sigma) \setminus \sigma, j \in \Delta(1) \setminus
\starr^1(\sigma))$, so $\Gprime \subseteq \L$ follows from
Part~\ref{i:Gasintersection} of Proposition~\ref{p:Gcone} and
Part~\ref{i:modifiedGcone} of Corollary~\ref{c:Lcone}.

Suppose now that $\Delta$ is pure of dimension $d$.  The inclusion $\L
\subseteq \U$ comes from Part~\ref{i:modifiedGconeforU} of
Proposition~\ref{p:Ucone}, since the intersection for $\U$ is over
only the $(d-1)$-dimensional cones of $\Delta$ rather than all cones 
as for $\L$.  The inclusion $\U \subseteq \Uw$ comes from the
definition of $\U$.  

If $\Delta$ is the fan of a projective toric variety then a divisor is
globally generated if and only if it is nef, so $\G=\nef(X_{\Delta})$.
For such $\Delta$ we have $d=n$, so $A_{n-d}(X_{\Delta}) \cong \mathbb
Z$, and thus $\U = \Uw = \{ D \in \Pic(X_{\Delta})_{\mathbb R} : [D]
\cdot [V(\tau)] \geq 0 \text{ for all } \tau \in \Delta(n-1) \}$.  The classes $\{V(\tau) : \tau \in \Delta(n-1) \}$ generate the Mori cone of curves of $X_{\Delta}$,
so we also have  $\U = \nef(X_{\Delta})$.  
\end{proof}

\section{Bounds for nef cones} \label{bounds}

In this section we prove the main theorem of this paper,
Theorem~\ref{t:bound}, which shows that given an appropriate embedding
of a projective variety $Y$ into $X_{\Delta}$, the pullbacks of the
cones $\G$ and  $\L$ give lower bounds for $\nef(Y)$, and the pullback of
the cone $\Uw$ for appropriate $\mathbf{w}$ gives an upper bound.

The upper bound requires some tropical geometry, which we first review briefly.

\begin{definition} \label{d:tropicalvariety}
Let $Y^0 \subset T$ be a subvariety of a torus $T\cong (\K^*)^n$.  Let
$K$ be an algebraically closed field extension of $\K$ with a
valuation $\val : K^* \rightarrow \mathbb R$ that is constant on $\K$
and with residue field isomorphic to $\K$.  The tropical variety
$\trop(Y^0)$ is equal as a set to the closure in $\mathbb R^n$ of 
$\{ (\val(y_1), \dots,\val(y_n) ) \in \mathbb R^n : (y_1,\dots,y_n) \in Y^0(K) \}$.
\end{definition}

By the structure theorem for tropical varieties, the set $\trop(Y^0)$
can be given the structure of a polyhedral fan of dimension
$\dim(Y^0)$ (\cite[Theorem 3.3.4]{TropicalBook}).  There are many
possible choices of fan structure.  We will fix one, which
we denote by $\Sigma'$, for which the initial ideal
$\inn_w(I(Y^0))$ is constant for all $w$ in the relative interior of a
cone, where $I(Y^0)$ is the ideal in $\K[T]$ defining $Y^0$.  For $w
\in \mathbb R^n$ the initial ideal $\inn_w(I(Y^0))$ is $\langle
\inn_w(f) : f \in I(Y^0) \rangle$, where for a Laurent polynomial
$f=\sum c_u x^u$ the initial form $\inn_w(f)$ is $\sum_{w \cdot u
  \text{ minimal } } c_u x^u$.

Given such a polyhedral fan structure $\Sigma'$ on $\trop(Y^0)$, we
associate to each top-dimensional cone $\sigma \in \Sigma'$ a positive
integer $m_{\sigma}$ as follows.  Choose $w$ in the relative interior
of $\sigma$; then $m_{\sigma} = \sum_{P \in \Ass(\inn_w(I(Y^0)))}
\mathrm{mult}(P, \inn_w(I(Y^0)))$, where $\Ass( \cdot)$ denotes the
set of associated primes of the ideal.  Let $\mathbf{m}=(m_{\sigma})$
be the vector of multiplicities as $\sigma$ varies over
top-dimensional cones.  The tropical variety of $Y^0$ is the pair
$(\trop(Y^0), \mathbf{m})$.  This is the ``constant coefficient'' case
of tropical geometry.  See \cite{TropicalBook} for more information.

Given any fan $\Delta$ with $|\Delta|=\trop(Y^0)$, where
$\dim(Y^0)=d$, we get an element $\mathbf{w} \in
\Hom(A_{n-d}(X_{\Delta}), \mathbb R)$ by setting
$\mathbf{w}(V(\sigma))$ equal to the weight $m_{\tau}$ on any
$d$-dimensional cone $\tau$ in $\Sigma'$ intersecting $\sigma$ in its
relative interior.  That this is well-defined follows from the
balancing condition on tropical varieties; see \cite[\S
  3.4]{TropicalBook}, \cite{FultonSturmfels}.

\begin{theorem} \label{t:bound}
Let $i \colon Y \rightarrow X_{\Delta}$ be an embedding of a
$d$-dimensional projective variety $Y$ into an $n$-dimensional normal
toric variety $X_{\Delta}$.  Then we have the following inclusions of
cones in $N^1(Y)_{\mathbb R}$:
$$i^*(\G) \subseteq i^*(\L)
\subseteq \nef({Y}). $$ 
If in addition we have
$\trop(Y^0) = |\Delta|$ for $Y^0 = Y \cap T$, and $i^* \colon
\Pic(X_{\Delta})_{\mathbb R} \rightarrow N^1(Y)_{\mathbb R}$ is
surjective, then
$$\nef(Y) \subseteq i^*(\Uw),$$ where $\mathbf{w} \in
\Hom(A_{n-d}(X_{\Delta}), \mathbb R)$ is as described above.
\end{theorem}

\begin{remark}
Conceptually the proof of this theorem is very straightforward.  The
upper bound comes from intersecting $D$ with curves obtained as the
intersection of $Y$ with appropriate torus-orbit closures on $\Delta$.
Care must be taken to make this rigorous however, as we do not assume
that $X_{\Delta}$ is smooth or complete.  In particular, we note that
if $\dim(\Delta)<n-1$ then $A_1(X_{\Delta})=A_0(X_{\Delta}) = 0$, so
the classes of these curves in $X_{\Delta}$ are zero.
\end{remark}

\begin{proof}
The inclusion $i^*(\G) \subseteq i^*(\L)$ was shown in
Proposition~\ref{p:toriccontainment}.  To see the inclusion $i^*(\L)
\subseteq \nef(Y)$, let $D' = i^*(D)$ for $D \in \L$.  Choose a
complete fan $\Sigma$ for which $\Delta$ is a subfan.  This is
possible by \cite[Theorem 2.8]{Ewald}.  Write $k \colon X_{\Delta}
\rightarrow X_{\Sigma}$ for the inclusion morphism.  After repeated
stellar subdivision of cones in $\Sigma \setminus \Delta$, we may
assume that if $\sum_{i \in \Delta(1)} a_i D_i$ is a Cartier divisor
on $X_{\Delta}$, then $\sum_{i \in \Delta(1)} a_i D_i$ is a Cartier
divisor on $X_{\Sigma}$, which implies that $k^* \colon
\Pic(X_{\Sigma}) \rightarrow \Pic(X_{\Delta})$ is surjective.  Let $C$
be an irreducible curve on $Y$, and let $\sigma \in \Delta$ be the
largest cone of $\Delta$ for which $C \subseteq V(\sigma)$, so $C \cap
\mathcal O_{\sigma} \neq \emptyset$, where $\mathcal O_{\sigma}$ is
the torus orbit corresponding to the cone $\sigma$.  Write
$V_{\Delta}(\sigma)$ and $V_{\Sigma}(\sigma)$ for the closure of
$\mathcal O_{\sigma}$ in $X_{\Delta}$ and $X_{\Sigma}$ respectively.
We may also regard $C$ as a curve on $X_{\Delta}$, on
$V_{\Delta}(\sigma)$, or on $V_{\Sigma}(\sigma)$.  We denote by
$[C]_Y$ the class of $C$ in $A_1(Y)$, and similarly for the other
ambient spaces.  Note that $j_*([C]_Y) =
{i_{\sigma}}_*([C]_{V_{\Sigma}(\sigma)}) = [C]_{\Sigma} \in
A_1(X_{\Sigma})$, where $j= k \circ i$.  Since $D \in \L$, by
Part~\ref{i:sigmapos} of Proposition~\ref{p:pisigmastar} we can write
$[D] = \sum a_i D_i$, where $a_i=0$ for $i \in \sigma$, and $a_i \geq
0$ for $i \in \starr^1(\sigma)$.  Let $\tilde{D} = \sum a_i D_i$ be
the corresponding divisor on $X_{\Sigma}$ with $k^*(\tilde{D})=D$.  We
then have the following diagram:

$$\xymatrix{
& V_{\Delta}(\sigma) \ar[r]^k \ar[d]^{i_{\sigma}} & V_{\Sigma}(\sigma) \ar[d]^{i_{\sigma}} \\
Y \ar[r]^i \ar@/_2pc/[rr]_{j} & X_{\Delta} \ar[r]^k & X_{\Sigma}
}
$$

Then $i_{\sigma}^*(\tilde{D}) = \sum_{\tau \in
  \starr^1_{\Sigma}(\sigma)} b_{\tau} D_{\tau}$, where $b_{\tau}$ is
as in Equation~\ref{e:pullback}.  Note that for $\tau \in \Delta$, the
coefficient $b_{\tau}$ is a positive multiple of $a_i$ for $i \in \tau
\setminus \sigma$, and thus $b_{\tau} \geq 0$.  Since $C \cap \mathcal
O_{\sigma} \neq \emptyset$, $D_{\tau} \cdot [C]_{V_{\Sigma}(\sigma)}
\geq 0$ for all torus-invariant prime divisors $D_{\tau}$ on
$V_{\Sigma}(\sigma)$.  In addition, since $C \subset Y \subset
X_{\Delta}$, we have $D_{\tau} \cdot [C]_{V_{\Sigma}(\sigma)} = 0$ for
$\tau \not \in \Delta$.  Thus $i_{\sigma}^*(\tilde{D}) \cdot
[C]_{V_{\Sigma}(\sigma)} \geq 0$, so by the projection formula 
${i_{\sigma}}_*(i_{\sigma}^*(\tilde{D}) \cdot [C]_{V_\Sigma(\sigma)})
= \tilde{D} \cdot [C]_{\Sigma} \geq 0$.  
Then $j_*(D' \cdot [C]_Y) = \tilde{D} \cdot [C]_{\Sigma} \geq 0$,
so since $j_*$ is not the zero map, $D' \cdot [C]_Y \geq 0$ and thus
$D' \in \nef(Y)$.

For the last inclusion, fix $D' \in \nef(Y)$.  Since $i^*$ is assumed
to be surjective we can write $D' = i^*(D)$ for some $D \in
\Pic(X_{\Delta})_{\mathbb R}$.  We wish to show that $\mathbf{w}(D
\cdot V(\tau)) \geq 0$ for all $\tau \in \Delta(d-1)$.

Choose a smooth toric resolution $X_{\widetilde{\Delta}}$ of
$X_{\Delta}$  with the property that if $\widetilde{Y}$ is the
closure of $Y^0$ in $X_{\widetilde{\Delta}}$ (so $\widetilde{Y}$ is
the strict transform of $Y$ with respect to the morphism $\pi \colon
X_{\widetilde{\Delta}} \rightarrow X_{\Delta}$), then the inclusion
$\widetilde{Y} \rightarrow X_{\widetilde{\Delta}}$ is tropical in the
sense of \cite{Tevelev}.  To see that this is  always possible,  begin
by taking the common refinement of $\Delta$ and a fixed tropical fan
with support $|\Delta|$, and then taking a smooth refinement of this
fan.  The resulting morphism $\pi \colon \widetilde{\Delta}
\rightarrow \Delta$ will be proper  and birational.
 Moreover, it defines a smooth tropical compactification as
 refinements of tropical fans are tropical by \cite[Proposition
   2.5]{Tevelev}.

Let $\widetilde{\mathbf{w}} \in \Hom(A_{n-d}(X_{\widetilde{\Delta}}),
\mathbb R)$ be the induced homomorphism defined by 
$\widetilde{\mathbf{w}}(V(\tilde{\sigma})) =
\mathbf{w}(\pi_*(V(\tilde{\sigma})))$ for all $\tilde{\sigma} \in
\widetilde{\Delta}(d)$.  To see that $\widetilde{\mathbf{w}}$ gives a
well-defined element of $\Hom(A_{n-d}(X_{\widetilde{\Delta}}),\mathbb
R)$ consider the relations on
$A_{n-d}(X_{\widetilde{\Delta}})$ given in \cite[Proposition
  2.1b]{FultonSturmfels}.  Given $\tau \in \Delta(d-1)$, choose
$\tilde{\tau}$ with $\pi(\tilde{\tau}) \subset \tau$, so
$\pi_*(V(\tilde{\tau})) = V(\tau)$ (\cite[p100]{Fulton}).  Then
$\mathbf{w}(D \cdot V(\tau)) = \mathbf{w}(\pi_*(\pi^*(D) \cdot
V(\tilde{\tau}))) = \widetilde{\mathbf{w}}(\pi^*(D) \cdot
V(\tilde{\tau}))$.  So it suffices to show that
$\widetilde{\mathbf{w}} (\pi^*(D) \cdot V(\tilde{\tau})) \geq 0$ for
all $\tilde{\tau} \in \widetilde{\Delta}(d-1)$.

Choose a smooth projective toric variety $X_{\widetilde{\Sigma}}$
whose fan $\widetilde{\Sigma}$ contains $\widetilde{\Delta}$ as a
subfan so $k \colon X_{\widetilde{\Delta}} \rightarrow
X_{\widetilde{\Sigma}}$ is an embedding.  This is possible, assuming a
suitably refined choice of $\widetilde{\Delta}$, because the tropical
fan structure on $|\Delta|$ given by the Gr\"obner fan has a
projective compactification given by the Gr\"obner fan, and we can
take a projective toric resolution of singularities of the resulting
fan, refining $\widetilde{\Delta}$ if necessary.  As before, choose
$\widetilde{\Sigma}$ so that if $\sum_{i \in \widetilde{\Delta}(1)}
a_i D_i$ is a Cartier divisor on $X_{\widetilde{\Delta}}$, then
$\sum_{i \in \widetilde{\Delta}(1)} a_i D_i$ is a Cartier divisor on
$X_{\widetilde{\Sigma}}$, and so the induced map $k^* \colon
\Pic(X_{\widetilde{\Sigma}}) \rightarrow \Pic(X_{\widetilde{\Delta}})$
is surjective.

Let $\tilde{j} \colon \widetilde{Y} \rightarrow X_{\widetilde{\Sigma}}$ be
the composition of $\tilde{i} \colon \widetilde{Y} \rightarrow
X_{\widetilde{\Delta}}$ and $k \colon X_{\widetilde{\Delta}}
\rightarrow X_{\widetilde{\Sigma}}$.  We thus have the following diagram.

$$\xymatrix{
\widetilde{Y} \ar[r]^{\tilde{i}} \ar[d]^{\pi} \ar@/^2pc/[rr]^{\tilde{j}}& 
X_{\widetilde{\Delta}} \ar[r]^k \ar[d]^{\pi} & X_{\widetilde{\Sigma}} \\
Y \ar[r]^i & X_{\Delta} & 
}$$

Fix $\tilde{\tau} \in \widetilde{\Delta}(d-1)$.  Since $\widetilde{Y}
\rightarrow X_{\widetilde{\Delta}}$ is a tropical compactification,
$C_{\tilde{\tau}} = \widetilde{Y} \cap V(\tilde{\tau})$ is a curve on $\widetilde{Y}$.
Let $C_1,\dots, C_r$ be the irreducible components of
$C_{\tilde{\tau}}$.  Then by \cite[\S 7.1]{Fulton} we have
$[\widetilde{Y}]_{\widetilde{\Sigma}} \cdot
[V_{\widetilde{\Sigma}}(\tilde{\tau})]_{\widetilde{\Sigma}} = \sum
  \lambda_i [C_i]_{\widetilde{\Sigma}} \in
    A_1(X_{\widetilde{\Sigma}})$, where $\lambda_i \geq 1$.

Write $\pi^*(D) = \sum_{i \in \widetilde{\Delta}(1)} a_i D_i$, and let
$\tilde{D}= \sum_{i \in \widetilde{\Delta}(1)} a_i D_i \in
A_{n-1}(X_{\widetilde{\Sigma}})$, so $k^*(\widetilde{D}) = \pi^*(D)$,
and thus $\pi^*(D') = {\tilde{j}}^*(\widetilde{D})$.  Since $D' \in
\nef(Y)$, we have $\pi^*(D') \in \nef(\widetilde{Y})$.  Thus
$\pi^*(D') \cdot [C_i]_{\widetilde{Y}} \geq 0$ for $1 \leq i \leq r$.
So $\tilde{D} \cdot [C_i]_{\tilde{\Sigma}} = \tilde{j}_* (\pi^*(D')
\cdot [C_i]_{\widetilde{Y}}) \geq 0$ by the projection formula.  Thus
$\tilde{D} \cdot [\widetilde{Y}] \cdot [V(\tilde{\tau})] \geq 0$,
where this computation takes place in $A^*(X_{\tilde{\Sigma}})$.

Extend $\widetilde{\mathbf{w}}$ to an element of
$\Hom(A_{n-d}(X_{\tilde{\Sigma}}),\mathbb R)$ by setting
$\widetilde{\mathbf{w}}(V(\sigma))=0$ for $\sigma \in
\widetilde{\Sigma}(d) \setminus \widetilde{\Delta}(d)$.  Again, this
is well-defined by \cite[Proposition 2.1b]{FultonSturmfels}.  Since
$\widetilde{Y} \rightarrow X_{\widetilde{\Delta}}$ is a smooth
tropical compactification, \cite[Lemma 3.2(1)]{SturmfelsTevelev}
implies that $[\widetilde{Y}] \cdot ([\tilde{D}] \cdot
[V(\tilde{\tau})]) = \widetilde{\mathbf{w}}([\tilde{D}] \cdot
[V(\tilde{\tau})])$.  Thus $\widetilde{\mathbf{w}}([\tilde{D}] \cdot
[V(\tilde{\tau})]) \geq 0$.

Note that if $\tilde{D} \cdot [V(\tilde{\tau})]_{\widetilde{\Sigma}} =
\sum_{\sigma \in \widetilde{\Sigma}(d)} a_{\sigma}
    [V(\sigma)]_{\widetilde{\Sigma}}$ for $\tilde{\tau} \in
    \widetilde{\Delta}(d-1)$ then $\pi^*(D) \cdot
              [V(\tilde{\tau})]_{\widetilde{\Delta}} = \sum_{\sigma
                \in \widetilde{\Delta}(d)} a_{\sigma}
              [V(\sigma)]_{\widetilde{\Delta}}$, so
              $\widetilde{\mathbf{w}}([\tilde{D}]_{\widetilde{\Sigma}}
              \cdot [V(\tilde{\tau})]_{\widetilde{\Sigma}}) =
              \widetilde{\mathbf{w}}([\pi^*(D)] \cdot
                        [V(\tilde{\tau})]) \geq 0$ as required.
\end{proof}

\begin{corollary}  \label{c:polyhedralnef}
Let $\Delta$ be a pure $d$-dimensional polyhedral fan such that there is 
$\mathbf{w} \in \Hom(A_{n-d}(X_{\Delta}),\mathbb R)$ with 
$\L = \Uw.$  Then any
subvariety $i \colon Y \rightarrow X_{\Delta}$ with $i^* \colon
\Pic(X_{\Delta})_{\mathbb R} \rightarrow N^1(Y)_{\mathbb R}$
surjective and $\trop(Y \cap T) = \Delta$ with multiplicities given by
$\mathbf{w}$ has $$\nef(Y) = i^*(\L) = i^*(\Uw).$$ The nef cone of
$Y$ is thus polyhedral.
\end{corollary}

\begin{proof}
We have the inclusions $i^*(\L) \subseteq \nef(Y) \subseteq i^*(\Uw)$
by Theorem~\ref{t:bound}, so $\L=\Uw$ implies that both inclusions are
equalities.  Since pullback is a linear map, and $\L$ and $\Uw$ are
both polyhedral, the conclusion follows.
\end{proof}

\begin{problem}
Characterize which polyhedral fans $\Delta$ have the property that
$\L = \Uw$ for some $\mathbf{w} \in \Hom(A_{n-d}(X_{\Delta}), \mathbb R)$.  When one
such $\mathbf{w}$ exists, characterize the set of possible
$\mathbf{w}$.  Also give a characterization of those fans  $\Delta$ with the property that
$\L = \Uw$ for all $\mathbf{w} \in \Hom(A_{n-d}(X_{\Delta}),\mathbb R)$.\end{problem}

\begin{remark}
While the hypothesis of Corollary~\ref{c:polyhedralnef} requires $\L =
\Uw$, it may be easier to check in examples the stronger condition
that $\G=\Uw$.
\end{remark}

\section{Examples} \label{Examples}
In this section we consider three families of examples of the cones
$\L$, $\G$, and $\Uw$, highlighting where they have previously appeared
in the literature in other guises.

In Section \ref{ss:delpezzo} we compute our cones for a family of del
Pezzo surfaces, illustrating two phenomena: firstly, in this case, the
closer the variety is to being toric, the closer the upper and lower
bounds are to each other and secondly, the bounds on the nef cone
given by our methods depend on the choice of toric embedding.  In
Section \ref{ss:CY} we relate our construction to the Cox/Katz
conjecture and examples of Buckley, Hassett/Lin/Wang, and
Szendr\H{o}i.  One of their examples shows that one can have $\G(Y)
\subsetneq \nef(Y)$; see Example~\ref{e:Bl2p4}.  Finally, in Section
\ref{ss:PicardRankTwo}, we compute our cones for subvarieties of
smooth Picard-rank two toric varieties.  Example~\ref{e:FDeltak} shows
that one does not always have $A_{n-d}(X_{\Delta}) \cong \mathbb Z$
when $\Delta$ is a $d$-dimensional fan in $\mathbb R^n$, and that the
cone $\Uw$ depends on the choice of $\mathbf{w}$.

The many calculations reported in this section were performed using a Macaulay 2 package
available from Maclagan's webpage \cite{GLUpackage}, \cite{M2}.

\subsection{Del Pezzo Surfaces}
\label{ss:delpezzo}

Del Pezzo surfaces are the blow-up of $\mathbb P^2$ in at most
eight general points.  We consider general del Pezzo surfaces,
which can be realized as tropical compactifications of complements of
particular line arrangements in $\mathbb P^2$.  We now review this
construction and the resulting cones $\mathcal G, \mathcal L$, and  $\mathcal
F_{\mathbf{w}}$.

Fix $1 \leq r \leq 8$ points $p_1,\dots,p_r$ in $\mathbb P^2$ such
that no three points lie on a line, no six points lie on a conic, and
no eight points lie on a cubic having a node at one of them.  If $1
\le r \le 3$, then $\overline{Y}=\mathrm{Bl}_{p_1,\dots,p_r}(\mathbb P^2)$ is a
projective toric variety, and by Proposition~\ref{p:toriccontainment},
the three cones all coincide with $\nef(\overline{Y})$.  For $r \ge 4$, we let
$\mathcal A$ be the line arrangement consisting of all ${ r \choose
  2}$ lines through pairs of the points.  We place the additional
genericity condition $(\dag)$ here on our choice of points that the
only intersection points of three or more lines are the original
$p_i$.  Such generic configurations exist for all $r$.  For $r \leq 5$
all configurations are generic in this sense, and for $r >5$ a
configuration with $r$ points can be obtained from one with $r-1$
points by choosing the $r$th point not on any line joining one of the
original $r-1$ points to any intersection point.

The line $\ell_{ij}$ joining $p_i$ to $p_j$ has the form $\{ y = (y_0
: y_1 :y_2) \in \mathbb P^2 : \mathbf{a}_{ij} \cdot y =
(\mathbf{a}_{ij})_0 y_0+ (\mathbf{a}_{ij})_1 y_1+ (\mathbf{a}_{ij})_2
y_2 = 0 \}$ for some $\mathbf{a}_{ij} \in \K^3$.  The complement $Y =
\mathbb P^2 \setminus \mathcal A$ embeds as a subvariety of the torus
$T^{{r \choose 2} -1}$ of $\mathbb P^{{r \choose 2}-1}$ via the map $\phi \colon Y
\rightarrow T^{{r \choose 2} -1}$ given by
$$\phi(y) = ( \mathbf{a}_{12} \cdot y : \dots : \mathbf{a}_{r-1 r}
\cdot y ) \in T^{{r \choose 2}-1}.$$ Let $A$ be the $3 \times {r \choose 2}$ matrix with
columns indexed by the pairs $\{i,j\}$ with $1 \leq i <j \leq r$ and
$ij$th column  $\mathbf{a}_{ij}$.  Then $\phi(Y) = \{ (z_{ij}) \in T^{{r \choose 2}-1} :
\sum b_{ij} z_{ij} = 0 \text{ for all } b=(b_{ij}) \in \ker(A) \}$.  

We denote by $\mathbf{e}_{ij}$ the basis vector for $\mathbb R^{r
  \choose 2}$ indexed by the pair $\{i,j\}$.  Let $\Sigma_r$ be the
two-dimensional fan in $\mathbb R^{r \choose 2} / \mathbb R(1,\dots,
1)$ whose rays are generated by the images of the points $\mathbf{e}_{ij}$
for $1 \leq i <j \leq r$ and the images of the points $\mathbf{f}_i =
\sum_{j \neq i} \mathbf{e}_{ij}$ for $1 \leq i \leq r$.  The cones of
$\Sigma_r$ are the images of $\pos(\mathbf{f}_i, \mathbf{e}_{ij}) +
\mathbb{R} (1,\dots, 1)$ for $1 \leq i \leq r$, $j \neq i$, and
$\pos(\mathbf{e}_{ij}, \mathbf{e}_{kl})+\mathbb R (1,\dots,1)$ for $\{
i,j, k, l \} \subset \{1,\dots,r\}$ with $|\{i,j,k,l\}|=4$ .

\begin{proposition} \label{p:delpezzo}
The tropical variety $\trop(\phi(Y)) \subseteq \mathbb R^{r \choose 2}
/ \mathbb R(1,\dots, 1)$ is the support of $\Sigma_r$.  The closure
$\overline{Y} \subset X_{\Sigma_r}$ equals the del Pezzo surface
$\mathrm{Bl}_{p_1,\dots,p_r}(\mathbb P^2)$, and the induced map $i^*
\colon \Pic(X_{\Sigma_r})_{\mathbb R} \rightarrow N^1(\overline{Y})_{\mathbb R}$ is
an isomorphism.  In addition $A_{{r \choose 2}-3}(X_{\Sigma_r}) \cong
\mathbb Z$.
\end{proposition}

\begin{proof}
The description of $\trop(\phi(Y))$ is immediate from the 
construction of the tropical variety of a linear space due
to \cite{ArdilaKlivans}; see also \cite[Chapter 4]{TropicalBook}.
Note that this is the coarse fan structure on this tropical variety.
The fact that the closure of $\phi(Y)$ in $X_{\Sigma_r}$ is
$\mathrm{Bl}_{p_1,\dots,p_r}(\mathbb{P}^2)$ follows from \cite[Example
  4.1]{Tevelev}, as our genericity assumptions on the $p_i$ ensure
that we are not in the exceptional case of that example.
To see that the induced map from $\Pic(X_{\Sigma_r})_{\mathbb R}$ is an
isomorphism, let $D_{f_i}$ be the divisor on $X_{\Sigma_r}$
corresponding to the ray through the image of $\mathbf{f}_i$, and let
$D_{ij}$ be the divisor corresponding to the ray through the image of
$\mathbf{e}_{ij}$.  We first note that $i^*(D_{f_i}) = E_i$, the
exceptional divisor obtained by blowing up the point $p_i$, and
$i^*(D_{ij}) = L_{ij}$, the strict transform of the line joining the
points $p_{i}$ and $p_j$. 
Since the $E_i$ and $L_{ij}$ span $N^1(\overline{Y})$, $i^*$ is
surjective.  Recall that $\{ E_i : 1 \leq i \leq r \} \cup \{ \ell\}$
form a basis for $N^1(\overline{Y})$, where $\ell$ is the pullback to
$\overline{Y}$ of a line in $\mathbb P^2$.  Thus to show the
isomorphism it suffices to show that $\Pic(X_{\Sigma_r}) \cong \mathbb
Z^{r+1}$.  This follow from the short exact sequence defining the
class group of $X_{\Sigma_r}$ \cite[p63]{Fulton}, since
$|\Sigma_r(1)|-({r \choose 2} -1) =r+1$.

The claim that $A_{{r \choose 2}-3}(X_{\Sigma_r}) \cong \mathbb Z$
follows from the description of the Chow groups of a toric variety
given in \cite[p337]{FultonSturmfels}.  
\end{proof}

We now list the relationships between $\nef(\overline{Y})$ and the
cones associated to $\Sigma_r$.

\begin{enumerate}

\item When $r=4$ we have $i^*(\mathcal G_{\Sigma_4})=i^*(\mathcal L_{\Sigma_4})
  = \nef(\overline{Y}) =i^*(\mathcal F_{\Sigma_4}) = i^*(\mathcal
  F_{\Sigma_4,\mathbf{w}})$.  In this case $\overline{Y} \cong \Bl_4(\mathbb P^2)
  \cong \overline{M}_{0,5}$.

\item When $r=5$ we have $i^*(\mathcal G_{\Sigma_5})=i^*(\mathcal
  L_{\Sigma_5}) = \nef(\overline{Y}) \subsetneq i^*(\mathcal
  F_{\Sigma_5}) = i^*(\mathcal F_{\Sigma_5,\mathbf{w}})$.  The cone
  $i^*(\mathcal L_{\Sigma_5})=\nef(\overline{Y})$ equals $i^*(\mathcal
  F_{\Sigma_5, \mathbf{w}} )\cap \{ [D] : [D] \cdot [C_5] \geq 0 \}$,
  where $[C_5]=2 \ell - \sum_{i=1}^5 E_i$ is the conic through the
  five points $p_1,\dots,p_5$.

\item When $r=6$ we have $i^*(\mathcal G_{\Sigma_6})=i^*(\mathcal L_{\Sigma_6})
  \subsetneq \nef(\overline{Y}) \subsetneq i^*( \mathcal F_{\Sigma_6})
  =i^*(\mathcal F_{\Sigma_6,\mathbf{w}})$.  As in the previous case
  $\nef(\overline{Y}) = i^*(\mathcal F_{\Sigma_6,\mathbf{w}}) \cap \{[D] :
  [D] \cdot [C_i] \geq 0, 1 \leq i \leq 6 \}$, where $C_i$ is the
  conic through the five points obtained by omitting $p_i$.  The cone
  $i^*(\mathcal L_{\Sigma_6})$ is the intersection of $\nef(\overline{Y})$
  (or indeed $i^*(\mathcal F_{\Sigma_6,\mathbf{w}})$) with $\{[D] : [C]
  \cdot [D] \geq 0 \}$, where $[C]=2\ell - \sum_{i=1} E_i$.  This
  class $[C]$ is not effective unless all six points lie on a conic,
  which will not be the case if the points are sufficiently general,
  so does not give a facet of $\nef(\overline{Y})$.

\item For $r=7,8$ we also have $i^*(\mathcal G_{\Sigma_r})=i^*(\mathcal
  L_{\Sigma_r}) \subsetneq \nef(\overline{Y}) \subsetneq i^*(\mathcal
  F_{\Sigma_r}) = i^*(\mathcal F_{\Sigma_r,\mathbf{w}})$.
\end{enumerate}

\subsection{Ample toric hypersurfaces} \label{ss:CY}

In this section we consider the case that $Y$ is a general ample
hypersurface in a simplicial projective toric variety $X_{\Sigma}$,
embedded by a morphism $i$.  Of particular interest is the case that
$X_{\Sigma}$ is Fano, and $[Y] = -K_{X_{\Sigma}}$, so $Y$ is
Calabi-Yau.  This situation was considered in \cite[\S6.2.3]{CoxKatz},
where they conjectured a description for $\nef(Y) \cap
i^*(\Pic(X_{\Sigma}))$.  We will show that this conjectured
description is $i^*(\G)$, where $\Delta$ is the subset of the
codimension-one skeleton of $\Sigma$ containing those cones
corresponding to torus orbits intersecting $Y$.  We first recall this,
and discuss the counterexamples given by Szendr\H{o}i and others in
our context.

A generalized flop of $X_{\Sigma}$ is a simplicial projective toric
variety with the same rays as $\Sigma$ whose fan is obtained by a
bistellar flip over a circuit $\Xi=(\Xi^+,\Xi^-)$ of $\Sigma$  (see
\cite[\S 7.2.C]{GKZ}), where bistellar flips are called modifications).
Loosely, this replaces cones containing $\pos(\mathbf{v}_i : i \in
\Xi^+)$ with those containing $\pos(\mathbf{v}_i : i \in \Xi^-)$.  A
generalized flop is a {\em trivial flip} if $|\Xi^+|, |\Xi^-| \geq 2$,
and $\cap_{i \in \Xi^-} D_i \cap Y = \emptyset$.  The first of these
conditions guarantees that $\Sigma'(1) = \Sigma(1)$, so
$\Pic(X_{\Sigma}) \cong \Pic(X_{\Sigma'})$.  We denote by $j^*$  the induced homomorphism $\Pic(X_{\Sigma'}) \rightarrow N^1(Y)_{\mathbb R}$.

In \cite[Conjecture 6.2.8]{CoxKatz} it was conjectured that $\nef(Y)
\cap i^*(\Pic(X_{\Sigma}))$ was equal to the union of
$j^*(\nef(X_{\Sigma'}))$ over all fans $\Sigma'$
that can be obtained from the fan $\Sigma$ by a sequence of trivial
flips.

\begin{lemma} \label{l:CoxKatzisG}
The cone $\bigcup j^*(\operatorname{Nef}(X_{\Sigma'}))$ equals $i^*(\G)$, where
$\Delta$ is the subset of the codimension-one skeleton of $\Sigma$ containing
those cones corresponding to torus orbit closures intersecting $Y$.
\end{lemma}

\begin{proof}
A general ample hypersurface $Y$ in $X_{\Sigma}$ will not pass through
the torus-fixed points of $X_{\Sigma}$, so by \cite[Lemma
  2.2]{Tevelev} the tropical variety $\trop(Y \cap T)$ is contained in
the codimension-one skeleton of $\Sigma$, and intersects 
those cones corresponding to torus-orbits closures intersecting $Y$.  Since
the tropical variety has dimension $n-1$ and is balanced, it must
actually be the union $\Delta$ of all such cones.

To show that $\bigcup i^*(\operatorname{Nef}(X_{\Sigma'})) = i^*(\G)$,
by Part~\ref{i:Gunionnef} of Proposition~\ref{p:Gcone}, it suffices to
show that a projective $\Sigma'$ can be obtained from $\Sigma$ by a
sequence of trivial flips if and only if $\Sigma'$ is the fan of
a projective toric variety with $\Sigma'(1)=\Delta(1)$ and $\Delta
\subseteq \Sigma'$.  For the ``only if'' direction, suppose that
$\Sigma'$ is obtained by a trivial flip over a circuit
$\Xi=(\Xi^+,\Xi^-)$ from a fan $\Sigma''$ containing $\Delta$.  From
the definition of trivial flip we have $\Sigma'(1)=\Sigma(1)$.  The
condition that $\cap_{i \in \Xi^-} D_i \cap Y = \emptyset$ implies
$\pos(\Xi^-) \not \in \Delta$, where $\pos(\Xi^-)$ is the cone
generated by rays indexed by $\Xi^-$.  Indeed, $\cap_{i \in \Xi^-}
D_i$ equals $V(\Xi^-)$, which is the orbit closure corresponding to
$\pos(\Xi^-)$.  This orbit closure intersects $Y$ if and only if
$\Delta$ contains a cone with $\pos(\Xi^-)$ as a face, and thus if and
only if $\pos(\Xi^-) \in \Delta$.  Since every cone removed from
$\Sigma''$ to obtain $\Sigma'$ contains $\pos(\Xi^-)$ as a face, we
conclude that no cone in $\Delta$ is removed, so $\Delta \subseteq
\Sigma'$.

For the ``if'' direction note that all fans $\Sigma'$ with
$\Sigma'(1)=\Delta(1)$ and $\Delta \subseteq \Sigma'$ are connected by
a sequence of bistellar flips through fans with the same property.
This follows from the chamber description of the secondary fan, since
the set of all chambers corresponding to $\Sigma'$ containing $\Delta$
form a convex cone; see \cite[Section 5.3]{TriangulationsBook}.  By
the above argument, since $\pos(\Xi^-) \not \in \Delta$ for these
bistellar flips, they are trivial flips.
\end{proof}

In \cite{SzendroiCoxKatz} Szendr\H{o}i gave a counterexample to
\cite[Conjecture 6.2.8]{CoxKatz}, which shows that $i^*(\G)$ can be a
strict lower bound for $\nef(Y)$ in this context.  This was followed
by other related examples by Buckley in \cite{Buckley}.  Both Hassett,
Lin, and Wang \cite{HassettLinWang}, and Szendr\H{o}i
\cite{SzendroiAmple} also consider the simpler, non-Fano, example
where $X_{\Sigma}$ is the blow-up of $\mathbb P^4$ at two points,
where the equality $\nef(Y) \cap i^*(\Pic(X_{\Sigma})) = i^*(\G)$ also fails.

\begin{example} \label{e:Bl2p4}
Let $X_{\Sigma}$ be the blow-up of $\mathbb P^4$ at two
torus-invariant points $p_1$ and $p_2$, whose fan $\Sigma$ has rays
spanned by the columns of the matrix $V$ below.  The corresponding
torus-invariant divisors are $\{D_0,D_1,D_2,D_3,D_4,E_1,E_2 \}$, where
$E_i$ for $i=1,2$ is the exceptional divisor of the blow-up at $p_i$.
The Picard group of $X_{\Sigma}$ is three-dimensional, with the classes
of the torus-invariant divisors given by the columns of the matrix $G$:
\renewcommand{\arraystretch}{0.8}
\renewcommand{\arraycolsep}{2pt}
$$V = \left( \text{\footnotesize $ \begin{array}{rrrrrrr}
-1 & 1 & 0 & 0 & 0 & 0 & 0 \\
-1 & 0 & 1 & 0 & 0 & 0 & 0 \\
-1 & 0 & 0 & 1 & 0 & 0 & -1 \\
-1 & 0 & 0 & 0 & 1 & -1 & 0 \\
\end{array}$}
\right), \, \, \, \, 
G= \left( \text{\footnotesize $ \begin{array}{rrrrrrr}
1 & 1 & 1 & 1 & 1 & 0 & 0 \\
0 & 0 & 0 &  0& 1 & 1 & 0 \\
0 & 0 & 0 & 1 & 0 & 0 & 1 \\
\end{array}$}
\right).
$$

The effective cone of $X_{\Sigma}$ is the positive orthant in these
coordinates, and is generated by $D_0=D_1=D_2$, $E_1$, and $E_2$.  The
nef cone is generated by $D_3$, $D_4$, and $F=(1,1,1) =
D_3+E_1=D_4+E_2$.  This is the triangle labelled $\mathcal B$ in \cite[p3]{SzendroiAmple}.

If $\Delta$ is the whole $3$-skeleton of $\Sigma$, then $\G =
\nef(X_{\Sigma}) = \pos(D_3, D_4, F)$, and $\L = \pos(D_3, D_4, F,
D_0)$, which is the union of the triangles labelled $\mathcal B$ and
$\mathcal C$ in \cite[p3]{SzendroiAmple}.  In this case $\Gprime=\L$.  

The group $A_1(X_{\Delta}) \cong A_{1}(X_{\Sigma}) \cong \mathbb Z^3$,
as it is dual to the class group, which is three-dimensional.  A basis
is given by $\{V(012), V(01p_1 ), V(01p_2) \}$, where by $V(01p_1)$ we
mean the orbit-closure corresponding to the three-dimensional cone
spanned by the rays corresponding to $D_0, D_1$, and $E_1$.  The set
$W = \{ \mathbf{w} \in \Hom(A_1(X_{\Delta}), \mathbb R) :
  \mathbf{w}(V(\sigma)) \geq 0 \text{ for all } \sigma \in \Delta(3)
  \}$ of Definition~\ref{d:Wdefn} is isomorphic to $\mathbb R_{\geq
    0}^3$ under the map that sends $\mathbf{w}$ to
  $(\mathbf{w}(V(012)), \mathbf{w}(V(01p_1)), \mathbf{w}(V(01p_2)))$.
  Explicitly, for $(\alpha,\beta,\gamma) \in \mathbb R^3_{\geq 0}$ we
  have:

\bigskip

\begin{tabular}{l|l}
$\mathbf{w}(V(\sigma))$ & $V(\sigma)$ \\ \hline
$\alpha$ & $V(01p_1)$, $V(02p_1)$, $V(03p_1)$, $V(12p_1)$, $V(13p_1)$, $V(23p_1)$ \\
$\beta$ & $V(01p_2)$, $V(02p_2)$, $V(04p_2)$, $V(12p_2)$, $V(14p_2)$, $V(24p_2)$ \\
$\gamma$ & $V(012)$\\
$\alpha+\gamma$ & $V(024)$, $V(014)$, $V(124)$ \\
$\beta+\gamma$ & $V(023)$, $V(013)$, $V(123)$ \\
$\alpha+\beta+\gamma$ & $V(034)$, $V(134)$, $V(234)$ \\\\
\end{tabular}

\bigskip 

The cone $\Uw$ is 
\begin{multline*}
\Uw= \pos((1,1,1), (\beta, \beta, -\alpha-\gamma),
(\alpha,-\beta-\gamma, \alpha) ) \\= \pos( F, \beta D_4 -(\alpha+\gamma)
E_2, \alpha D_3 -(\beta+\gamma)E_1).
\end{multline*} 
The intersection is 
$$\U = \cap_{\mathbf{w} \in W} \Uw = \pos((1,1,1), (1,0,1), (1,1,0), (1,0,0)
) = \mathcal B \cup \mathcal C.$$
Note, however that for any one fixed
$\mathbf{w}$ we have $\Uw \not \subseteq \Eff(X_{\Sigma})$.  This example is illustrated in cross-section in Figure~\ref{f:Bl2P4}.
\begin{figure}[h]
\center{\resizebox{10cm}{!}{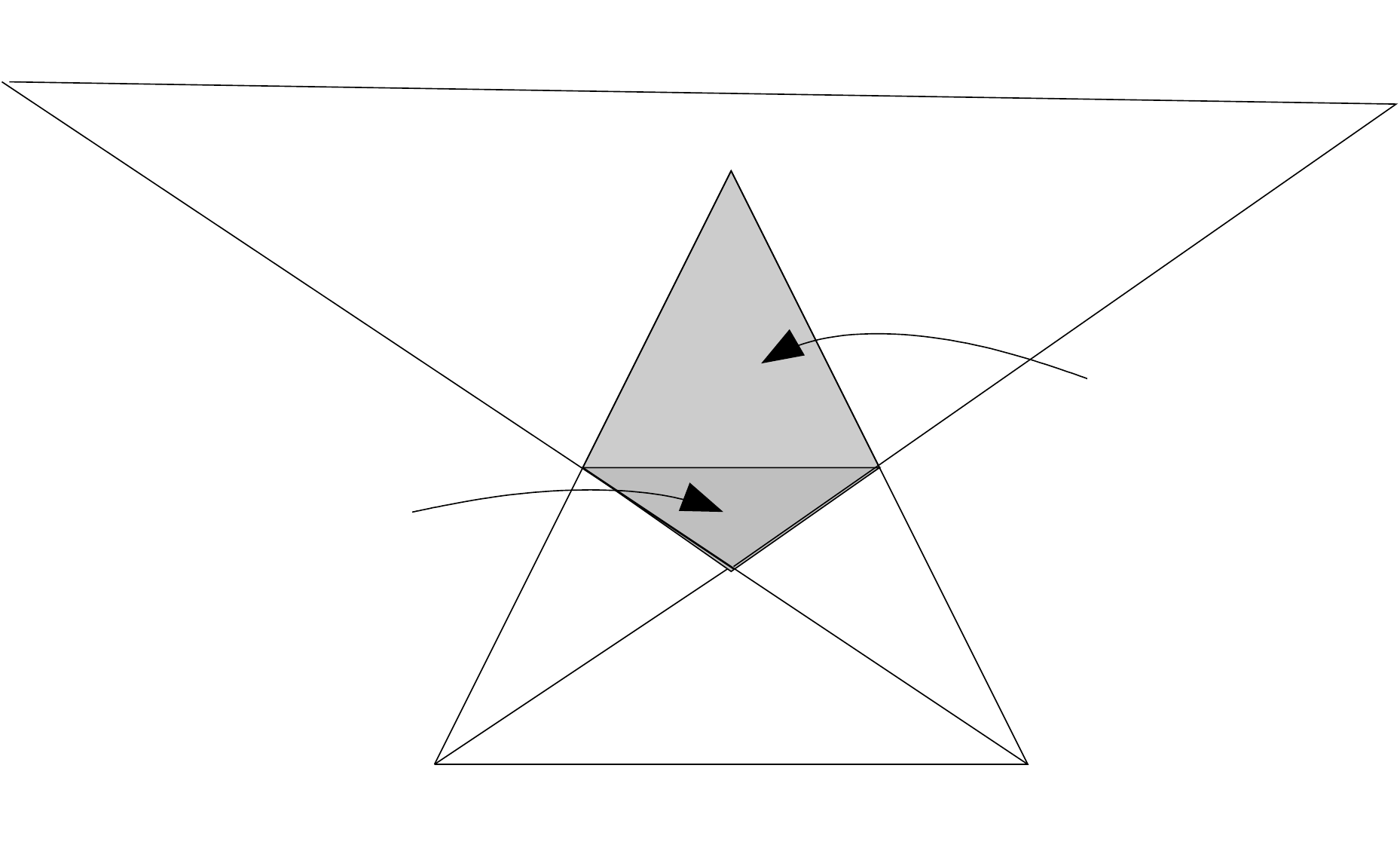}}
\caption{\label{f:Bl2P4}}
\end{figure}
\end{example}

Let $D$ be the image of $Y$ in $\mathbb P^4$ under the blow-down map.
We assume that $Y$ is the strict transform of $D$, and that $D$
contains $p_1$ and $p_2$.  Let $L$ be the line joining $p_1$ and $p_2$
in $\mathbb P^4$.  In \cite{HassettLinWang} Hassett, Lin, and Wang
show that if $D$ intersects $L$ in three distinct points, and contains
a finite number of lines meeting these three points then $\nef(Y) =
\mathcal B \cup \mathcal C$.  This is extended in
\cite{SzendroiAmple}, where Szendr\H{o}i classifies when $\nef(Y)$ is
$\mathcal B$ or $\mathcal B \cup \mathcal C$.  Szendr\H{o}i shows that
if $D$ contains the line $L$ then $\nef(Y) = \mathcal B$, if $D$ does
not contain $L$, but for each $i=1,2$ either $D$ contains a line
through $p_i$ or $L$ is tangent to $D$ at $p_i$ then $\nef(Y) =
\mathcal B \cup \mathcal C$, and in all other cases there is a strict
inclusion $\mathcal B \cup \mathcal C \subsetneq \nef(Y)$.  In the
first of these cases $Y$ contains some torus-fixed points of
$X_{\Sigma}$.  If we add the corresponding cone to $\Delta$ we get $\G
= \L = \mathcal B$.
The last case shows that $\U$ (as opposed to $\Uw$) is not always an upper
bound for $\nef(Y)$ when the rank of $A_{n-d}(X_{\Delta})$ is greater
than one.

The details of this example, and the others in \cite{SzendroiCoxKatz},
\cite{Buckley}, \cite{SzendroiAmple}, and \cite{HassettLinWang} are
contained in the Macaulay 2 package \cite{GLUpackage}.

\subsection{Subvarieties of Picard rank two smooth toric varieties}
\label{ss:PicardRankTwo}

In this section we consider subvarieties of smooth toric varieties
with rank-two Picard groups.  A projective such $X_{\Sigma}$ is
determined by integers $s \geq 2$ and a sequence $0 \leq a_1 \leq a_2
\leq \dots \leq a_r$ as follows.  Let $\mathbf{e}_1,\dots,
\mathbf{e}_n$ be a basis for $\mathbb R^n$, for $n=r+s-1$.  We may
choose coordinates so that $\Sigma$ has rays spanned by $\mathbf{v}_0=
-\sum_{i=1}^{s-1} \mathbf{e}_i+\sum_{j=1}^r a_j \mathbf{e}_{j+s-1}$,
$\mathbf{v}_i=\mathbf{e}_i$ for $1 \leq i \leq s-1$,
$\mathbf{u}_0=-\sum_{j=1}^r \mathbf{e}_{j+s-1}$, and $\mathbf{u}_j =
\mathbf{e}_{j+s-1}$ for $1 \leq j \leq r$.  The cones of $\Sigma$ are
then $\sigma_{ij} = \pos(\mathbf{v}_k, \mathbf{u}_l : k \neq i, l \neq
j )$ for $0 \leq i \leq s-1$, $0 \leq j \leq r$.  See
\cite{Kleinschmidt} for details.

The Picard group of $X_{\Sigma}$ is isomorphic to $\mathbb Z^2$.  We
denote by $D_i$ the divisor corresponding to the ray through
$\mathbf{v}_i$, and $E_i$ the divisor corresponding to the ray through
$\mathbf{u}_i$.  We can choose a basis of $\mathbb Z^2$ so that $[D_i]
= (1,0)$ for $0 \leq i \leq s-1$, and $[E_i]=(-a_i,1)$, where we set
$a_0=0$.  The nef cone of $X_{\Sigma}$ is then the positive orthant
$\pos((1,0), (0,1))$.  See Figure~\ref{f:codim2}.

\begin{figure}
\center{\resizebox{3.5cm}{!}{\includegraphics{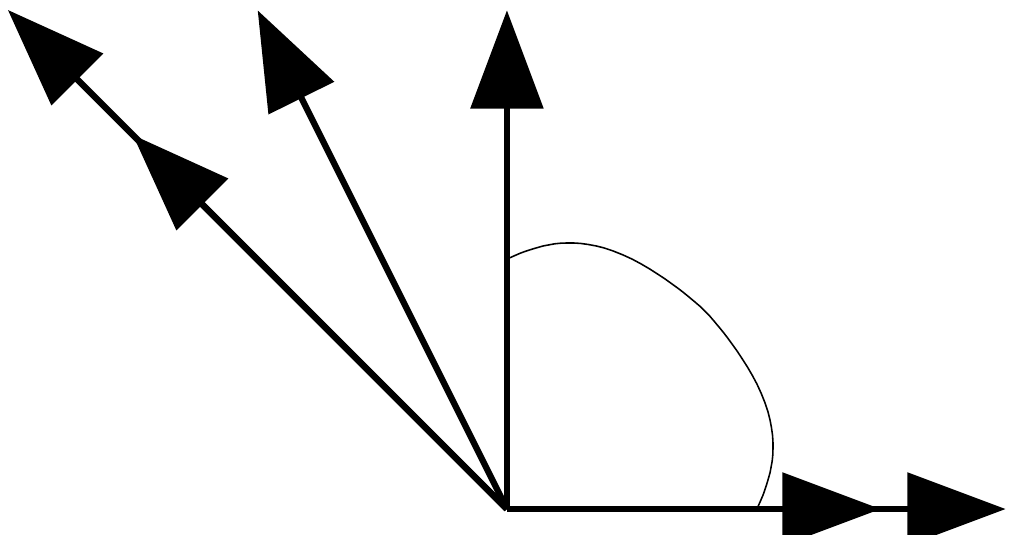}}}
\caption{\label{f:codim2}}
\end{figure}

Let $\Delta_k$ be the subfan of $\Sigma$ consisting of just those
cones of $\Sigma$ of dimension at most $k$ for fixed $k>0$.  A generic
$k$-dimensional subvariety of $X_{\Sigma}$ is contained in
$X_{\Delta_k}$.  We then have the following relation between the
cones $\mathcal G_{\Delta_k}$ and  $\mathcal L_{\Delta_k}$.

\begin{proposition} \label{p:codim2GL}
The cone $\mathcal G_{\Delta_k}$ equals $\pos( (1,0), (-a_i,1))$, where $i=0$ if $k
\geq r$, and $r-k$ if $k<r$.  The cone $\mathcal L_{\Delta_k}$ equals
$\mathcal G_{\Delta_{k-1}}$.
\end{proposition}

\begin{proof}
The cones of $\Delta_k$ are indexed by a pair $\{I, J\}$ with $I
\subset \{0,\dots,s-1\}$ and $J \subset \{0,\dots,r\}$ both nonempty and
$|I|+|J| \geq r+s+1-k$.  The corresponding cone is $\sigma_{IJ} =
\pos(\mathbf{v}_i, \mathbf{u}_j : i \not \in I, j \not \in J)$.  Thus
$\mathcal G_{\Delta_k} = \bigcap_{(I,J)} \pos( D_i, E_j : i \in I, j
\in J )$. The cone $\pos(D_i, E_j : i \in I, j \in J)$ equals
$\pos(D_0, E_k)$, where $k = \max \{ j : j \in J \}$.  Thus $\mathcal
G_{\Delta_k} = \pos(D_0, E_l)$, where $l = \min_{\{I,J\}} \max \{ j: j
\in J \}$ and the minimum is taken over all pairs $\{I,J\}$ with $I,J
\neq \emptyset$ and $|I|+|J| \geq r+s+1-k$.  When $k\geq r$ the
minimum is achieved at $I= \{ 0,\dots,r+s-k-1 \}$ and $J=\{0\}$, while
for $k <r$ the minimum is achieved at $I=\{0,\dots,s-1\}$ and
$J=\{0,\dots, r-k\}$, which implies the result.

The star of $\sigma_{IJ}$ contains all rays of $\Delta_k$ unless
$|I|=1$ or $|J|=1$, in which case the ray labelled by the singleton is
not contained in the star.  Thus $\mathcal L_{\Delta_k}$
equals the intersection of $\bigcap_{\{I,J\}} \pos(D_i, E_j : i \in I,
j \in J)$ with $\bigcap_{i \in \{0,\dots,s-1\}} \bigcap_{J \subset
  \{0,\dots,r\}} \pos(\pm D_i, E_j : j \in J)$ and $\bigcap_{I \subset
  \{0,\dots,s-1\}} \bigcap_{j \in \{0,\dots, r\}} \pos(D_i, \pm E_j :
i \in I)$, where the first intersection is over pairs $\{I,J\}$ with
$|I|, |J| \geq 2$ and $|I|+|J| \geq r+s-k$, and in the second and
third intersections we have $|J|\geq r+s-1-k$ and $|I| \geq r+s-1-k$.
Since $\pos(\pm D_i, E_j : j \in J)$ is the upper half plane for all
$j$, the second intersection is the upper half plane.  For the third
intersection, $\pos(D_i, \pm E_j : i \in I)$ is the halfspace $\{(x,y)
: x+a_jy \geq 0 \}$, so the intersection over all $j$ is the 
cone $\pos(D_0,-E_r)$.  Note that this case only occurs when $s+1
\geq r+s-k$, so $k\geq r+1$.  As the description of the first
intersection is as in the previous paragraph, we conclude that
$\mathcal L_{\Delta_k} = \pos(D_0,E_i)$, where $i=0$ for $k \geq r+1$,
and $i=r-k$ if $k \leq r$, so $\mathcal L_{\Delta_K}=\mathcal
G_{\Delta_{k-1}}$.
\end{proof}

\begin{example} \label{e:FDeltak}
We compute the cone $\mathcal F_{\Delta_k,\mathbf{w}}$ for $r=1$ and $s=n$.  The group $A_{n-k}(X_{\Delta_k}) \cong
\mathbb Z^2$.  One way to see this is to note that
$A_{n-k}(X_{\Delta_k}) \cong A_{n-k}(X_{\Sigma})$, since the
generators and the relations for this group only depend on $\Delta_k$,
and by \cite[p106]{Fulton} this is equal to the
degree-$k$ part of the ring $\mathbb Z[D_0,E_0]/\langle D_0^n,
E_0(E_0-a_1D_0) \rangle$, which is rank two.  We choose the basis
$D_0^k, D_0^{k-1}E_0$ for this group, and choose $\mathbf{w}=(w_1,w_2) \in
\Hom(A_{n-k}(\Delta_k), \mathbb R) \cong \mathbb R^2$.  The classes of
$V(\sigma)$ for $\sigma \in \Delta_k(k)$ are $\{ D_0^k, D_0^{k-1}E_0,
D_0^{k-1}(E_0-a_1D_0) \}$ in this basis, so we must choose $w_1 \geq 0, w_2 \geq
a_1w_1 \geq 0$.  Now
\begin{align*}
\mathcal F_{\Delta_k, \mathbf{w}} & = \{ D = aD_0+bE_0 : \mathbf{w}(D \cdot
V(\tau)) \geq 0 \text{ for all } \tau \in \Delta_k(k-1) \} \\ & = \{
aD_0 + bE_0 : \mathbf{w}(a D_0^k + b D_0^{k-1}E_0) \geq 0,
\mathbf{w}((a+a_1b)D_0^{k-1}E_0) \geq 0, \\
& \hspace{1cm} \mathbf{w}(aD_0^{k-1}(E_0-a_1D_0)) \geq 0 \} \\
& = \{ aD_0 +bE_0 : aw_1+bw_2 \geq 0, (a+a_1b)w_2 \geq 0,  a(w_2-a_1w_1) \geq 0 \}\\
& = \{ aD_0 + bE_0 : aw_1+bw_2 \geq 0, a+a_1b \geq 0, a \geq 0 \},\\
\end{align*}
where the last equality comes from the fact that $w_2, w_2-a_1w_1 \geq
0$.  The middle inequality is redundant, so this is the cone spanned
by $E_0$ and $w_2D_0-w_1E_0$.  If $k \geq 2$ we have $\mathcal
G(\Delta_k) = \mathcal L(\Delta_k) = \pos(E_0, D_0)$ by
Proposition~\ref{p:codim2GL}, so the inequality $\mathcal L(\Delta_k)
\subseteq \mathcal F_{\Delta_k,w}$ is strict unless $w_1=0$.  
\end{example}

\begin{remark}
Example~\ref{e:FDeltak} shows that one does not always have
$A_{n-d}(X_{\Delta}) \cong \mathbb Z$ for $\Delta$ a $d$-dimensional
fan in $\mathbb R^n$, and that the cone $\Uw$ depends on the choice of $\mathbf{w}$.
The intersection $\mathcal F_{\Delta_k}$ of all $\mathcal
F_{\Delta_k,\mathbf{w}}$ as $\mathbf{w}$ varies is $\pos(D_0,E_0) = \mathcal
G_{\Delta_k} = \mathcal L_{\Delta_k}$.  
\end{remark}

\section{Mori dream spaces}
\label{s:MDS}

In this section we show that if $Y$ is a Mori dream space then there
is an embedding $i \colon Y \rightarrow X_{\Delta}$ for which $i^*(\G)
= \nef(Y)$.  Recall from \cite{HuKeel} that a projective $\mathbb Q$-factorial
variety $Y$ with $\Pic(Y) \cong \mathbb Z^r$ is a Mori dream space if
the Cox ring
$$\Cox(Y) = \bigoplus_{u \in \mathbb Z^r} H^0(Y, L_1^{\otimes
  u_1}\otimes \dots \otimes L_r^{\otimes u_r})$$ is finitely
generated, where $L_1,\dots, L_r$ form a basis for $\Pic(Y)$.
Important examples include log Fano varieties (see \cite{BCHM}).

The ring $\Cox(Y)$ has a $\mathbb Z^r$-grading given by the Picard
group of $Y$.  Choose a graded presentation for $\Cox(Y)$:
$$\Cox(Y) \cong \K[z_1,\dots,z_N]/I,$$ where $I$ is homogeneous with
respect to the $\Pic(Y)$ grading.  We denote by $V(I)$ the affine
subscheme of $\mathbb A^N$ defined by the ideal $I$.  The action of
the torus $T = \Hom(\Pic(Y), \K^*) \cong (\K^*)^r$ on $\mathbb A^N$
descends to an action on $V(I)$.  Linearizations of this action
correspond to characters of $T$, and thus to line bundles on $Y$.  If
$L$ is an ample line bundle, then 
\begin{equation} \label{e:CoxGIT}
Y = \Proj(\oplus_{k\geq 0}
\Cox(Y)_{kL}) = V(I) \git_{L} T.
\end{equation}
 This gives an embedding $i: Y = V(I) \git_LT \rightarrow \mathbb A^N
 \git_L T$.  This latter space is a normal toric variety which we
 denote by $X_{\Sigma}$.  Let $\Delta$ be the subfan of $\Sigma$
 containing those cones for which the corresponding $T$-orbit closure
 intersects $Y$.  The embedding $i$ restricts to an embedding $i: Y
 \rightarrow X_{\Delta}$.

\begin{proposition}\label{MDS}
Let $Y$ be a Mori Dream Space, and let $i \colon Y \rightarrow
X_{\Delta}$ be the toric embedding described above.  Then 
$$\nef(Y) = i^*(\G) = i^*(\L).$$
\end{proposition}

\begin{proof}
It suffices to show that $\nef(Y) = i^*(\G)$, as the equality with
$i^*(\L)$ then follows from Theorem~\ref{t:bound}.  In
\cite[Proposition 2.9]{HuKeel} Hu and Keel show that the description
of Equation~\ref{e:CoxGIT} satisfies the condition of \cite[Theorem
  2.3]{HuKeel}, and so the Mori chambers of $Y$ are equal to the GIT
chambers.  Since $\nef(Y)$ is a Mori chamber, the proof reduces to
showing that $i^*(\G)$ is a GIT chamber of the GIT description for $Y$
given in Equation~\ref{e:CoxGIT}.  

For $L \in \Pic(Y)$ we denote by $V(I)^{ss}_L$ the semistable locus
for $V(I)$ with respect to the linearization of $T$ labelled by $L$.
The GIT chamber of $Y$ corresponding to $L$ is $\pos(\{ L' :
V(I)^{ss}_{L'} = V(I)^{ss}_L \}) \subseteq \Pic(Y)_{\mathbb R}$.  Now
$V(I)^{ss}_L$ is the intersection of $V(I)$ with $(\mathbb
A^N)^{ss}_L$.  The GIT chamber of $\mathbb A^N \git T$ with respect to
the linearization by $L$ equals the nef cone of the toric variety
$\mathbb A^N \git_L T$, whose fan we denote by $\Sigma_L$.  Thus the
GIT chamber of $V(I)  \git T$ containing a fixed ample $L$ is the union of
the nef cones of those toric varieties $X_{\Sigma_{L'}}=\mathbb A^N
\git_{L'} T$ for which $(\mathbb A^N)^{ss}_{L'} \cap V(I) = (\mathbb
A^N)^{ss}_{L} \cap V(I)$.  To show that this equals $i^*(\G)$ it
suffices by Part~\ref{i:Gunionnef} of Proposition~\ref{p:Gcone} to
show that this condition is equivalent to $\Delta \subseteq
\Sigma_{L'}$.

Let $\mathbf{v}_i$ be the first lattice point on the $i$th ray of $\Sigma_L$.
If $\sigma \in \Delta$ is a maximal cone, then $Y \cap \mathcal
O(\sigma) \neq \emptyset$, so then there is $x \in (\mathbb A^N)^{ss}
\cap V(I)$ with $\sigma = \pos ( \mathbf{v}_i : x_i =0)$.  Now note that for
  any $L'$ with $(\mathbb A^N)^{ss}_{L'} \cap V(I) = (\mathbb
  A^N)^{ss}_{L} \cap V(I)$ this $x$ lies in $(\mathbb A^N)^{ss}_{L'}$, so 
$\pos ( \mathbf{v}_i : x_i =0 ) \in \Sigma_{L'}$.  Thus $\Delta \subseteq \Sigma_{L'}$.  

Conversely, suppose that $\Delta \subseteq \Sigma_{L'}$ for some $L'$.
If $x \in (\mathbb A^N)^{ss}_{L} \cap V(I)$, then $\pos(\mathbf{v}_i :
x_i=0 ) \in \Delta$, so $\pos(\mathbf{v}_i : x_i =0) \in \Sigma_{L'}$.
This means that $x \in (\mathbb A^N)^{ss}_{L'}$, 
 and so $(\mathbb A^N)^{ss}_{L} \cap V(I) \subseteq
(\mathbb A^N)^{ss}_{L'} \cap V(I)$.   

 Now $Y=V(I) \git_{L}T$ is the closure of $Y \cap T$ in
 $X_{\Sigma_L}$, and similarly $Y'= V(I) \git_{L'} T$ is the closure
 of $Y \cap T$ in $X_{\Sigma_{L'}}$.  Since the closure of $Y \cap T$
 in $X_{\Sigma_L}$ is contained in $X_{\Delta}$, the same must be true
 for the closure of $Y \cap T$ in $X_{\Sigma_{L'}}$.  If, however,
 this inclusion $(\mathbb A^N)^{ss}_{L} \cap V(I) \subseteq (\mathbb
 A^N)^{ss}_{L'} \cap V(I)$ were proper, then there would be a point $x
 \in (\mathbb A^N)^{ss}_{L'} \cap V(I)$ with $\sigma
 =\pos(\mathbf{v}_i : x_i=0)$ satisfying $\sigma \in \Sigma_{L'}$,
 $\sigma \not \in \Delta$, and $(V(I) \git_{L'} T ) \cap \mathcal
 O(\sigma) \neq \emptyset$, where the intersection takes place in
 $X_{\Sigma_{L'}}$.  From this contradiction we conclude that
 $(\mathbb A^N)^{ss}_{L} \cap V(I) = (\mathbb A^N)^{ss}_{L'} \cap
 V(I)$.  Thus $(\mathbb A^N)^{ss}_{L} \cap V(I) \subseteq (\mathbb
 A^N)^{ss}_{L'} \cap V(I)$ if and only if $\Delta =
 \Sigma_{L'}$ as required.
\end{proof}

\begin{remark}
We note that this combinatorial description of the nef cone of a Mori
dream space was already known, using other language, by Berchtold and
Hausen; see \cite{BerchtoldHausen}.
\end{remark}

\begin{example}\label{embeddingexample}
Let $Y = \Bl_6(\mathbb P^2)$ be the blow-up of $\mathbb P^2$ at six
general points $p_1,\dots,p_6$.  This is a del Pezzo surface of
degree $3$.  The Cox ring of $Y$ has a generator of degree $E_i$ for
each of the six exceptional divisors $E_i$, one of degree $\ell
-E_i-E_j$ (the strict transform of the line joining $p_i$ and $p_j$)
for each of the fifteen such lines, and a generator of degree $2\ell -
\sum_{k \neq i} E_k$ for each of the six conics through five of the
points.  The ideal $I$ of relations is generated in degree $2$ (see for example 
\cite{StillmanTestaVelasco}).  We thus get an embedding $i \colon Y
\rightarrow X_{\Delta} \subset X_{\Sigma} = \mathbb A^{27} \git
(\K^*)^7$ of the surface $Y$ into a $20$-dimensional toric variety
$X_{\Delta}$ with $27$ rays and $\nef(Y) = i^*(\G) = i^*(\L)$.

This contrasts with the embedding of Section~\ref{ss:delpezzo}, where
$Y$ is embedded into a $14$-dimensional toric variety with $20$ rays
and $i^*(\G)=i^*(\L) \subsetneq \nef(Y)$.  The missing generators in
this case correspond to the conics through sets of five points.  This
example illustrates that the bound on $\nef(Y)$ obtained depends on
the choice of toric embedding.
\end{example}

\section{Bounds for the nef cone of  $\MOn$}
\label{s:M0n}
In this section we apply the main theorem to obtain bounds for the nef
cone of the moduli space $\MOn$ of stable $n$-pointed curves of genus
zero.  Kapranov's construction of $\MOn$ as a Chow or Hilbert quotient
of the Grassmannian $G(2,n)$ by an algebraic torus
\cite{KapChow},\cite{KapVer} gives rise to a natural embedding of the
moduli space into a toric variety $X_{\Delta}$ \cite{Tevelev},
\cite{GMEquations}.  In Proposition~\ref{p:GLF} we give simple and
explicit descriptions of the three corresponding cones of divisors
$\mathcal{G}_{\Delta}(\MOn)$, $\mathcal{L}_{\Delta}(\MOn)$, and
$\mathcal{F}_{\Delta}(\MOn)$ that give lower and upper bounds for
$\nef(\MOn)$.
We also show that $\mathcal{F}_{\Delta}(\MOn)$ is the cone of
  $\operatorname{F}$-divisors, which the $\operatorname{F}$-Conjecture
  asserts is equal to $\nef(\overline{\operatorname{M}} _{0,n})$.  Finally, we
  propose that the cone $\mathcal{L}_{\Delta}(\MOn)$ is an equally
  likely, and useful,  polyhedral description of $\nef(\MOn)$.  

We first recall the $\operatorname{F}$-Conjecture.  See, for example,
\cite{invitation} for further background on $\MOn$.  An
     {\em{$\operatorname{F}$-curve}} on $\MOn$ is any curve that is
     numerically equivalent to a component of the locus of points in
     $\MOn$ corresponding to curves having at least $n-4$ nodes.  An
     {\em $\operatorname{F}$-divisor} on $\MOn$ is any divisor that
     nonnegatively intersects every $\operatorname{F}$-curve.  The
     {\em{$\operatorname{F}$-Conjecture}} on $\MOn$ says that a
     divisor is nef if and only if it is an
     $\operatorname{F}$-divisor.  The $\operatorname{F}$-conjecture
     can be stated for $\overline{M}_{g,n}$ for all $g$, and the case
     $g>0$ was shown in \cite{GKM} to be implied by the
     $\operatorname{F}$-conjecture for $S_g$-symmetric divisors on
     $\overline{M}_{0,g+n}$.

 In order to state Proposition \ref{p:GLF}, we use the following
 simplicial complex $\widetilde{\Delta}$.

\begin{definition} \label{d:R} 
Let $\mathcal{I}=\{I \subset \{1,\ldots,n\}: 1 \in I \text{ and }
|I|,|I^c| \geq 2\}$.  Let $\widetilde{\Delta}$ be the simplicial
complex on the vertex set $\mathcal I$ for which  $\sigma \in
\widetilde{\Delta}$ 
if for all $I,J \in \sigma$ we have $I \subseteq J$, $J \subseteq I$,
or $I \cup J = \{1,\ldots,n\}$.
\end{definition}

 The maximal cones of $\Delta$ have dimension $n-3$, and the simplices
of $\widetilde{\Delta}$ are in bijection with boundary
strata of $\overline{\operatorname{M}} _{0,n}$.  
 See, for example, \cite{ArbarelloCornalbaSurvey} or \cite{invitation}
 for a description of the boundary divisors $\delta_I$ on $\MOn$.

\begin{proposition}\label{p:GLF}  
  There are three cones $\mathcal{G}_{\Delta}(\MOn)$,
  $\mathcal{L}_{\Delta}(\MOn)$, and $\mathcal{F}_{\Delta}(\MOn)$ in
  $N^1(\MOn)_{\mathbb R}$ that bound $\nef(\MOn)$:
 $$\mathcal{G}_{\Delta}(\MOn) \subseteq \mathcal{L}_{\Delta}(\MOn)
  \subseteq \nef(\MOn) \subseteq \mathcal{F}_{\Delta}(\MOn).$$ These
  are described as follows.
\begin{enumerate}
\item $\mathcal{G}_{\Delta}(\MOn) = \bigcap_{\sigma \in  \widetilde{\Delta} } \pos(\delta_I : I  \in \mathcal{I} \setminus \sigma)$.
\item 
 $\mathcal{L}_{\Delta}(\MOn) =\bigcap_{\sigma \in
    \widetilde{\Delta} } \pos(\delta_I , \pm \delta_J : I, J \in \mathcal{I}
  \setminus \sigma,  \delta_I \cap \delta_K \ne \emptyset, \\   \forall
  K \in \sigma, \text{and } \delta_J \cap \delta_L = \emptyset \mbox{
    for some } L \in \sigma)$.
\item $\mathcal{F}_{\Delta}(\MOn) =
  \bigcap_{\sigma \in \widetilde{\Delta}, |\sigma|=n-4} \pos(\delta_I , \pm \delta_J : I,
  J \in \mathcal{I} \setminus \sigma,  \delta_I \cap \delta_K
  \ne \emptyset, \\
 \forall K \in \sigma, \text{and } \delta_J \cap
  \delta_L = \emptyset  \text{ for some } L \in \sigma)$.
\end{enumerate}
Moreover, $\mathcal{F}_{\Delta}(\MOn)$ is equal to the cone of
$\operatorname{F}$-divisors on $\MOn$.
\end{proposition}

The key to proving Proposition~\ref{p:GLF} is to recognize
$\widetilde{\Delta}$ as the simplicial complex corresponding to a fan
$\Delta \subset \mathbb R^{{n \choose 2}-n}$.  This fan, known as the
space of phylogenetic trees, arises in the consideration of $\MOn$ as
a Chow or Hilbert quotient, and there is an embedding of $\MOn$ into
the associated toric variety $X_{\Delta}$.  We summarize the necessary
information in the following proposition; see \cite{Tevelev} or
\cite[\S 5]{GMEquations} for more information.

\begin{proposition} \label{p:DeltaDef}
There is a collection $\{ \mathbf{r}_I : I \in \mathcal I\}$ of
lattice points in $\mathbb R^{{n \choose 2}-n}$ for which the
collection of cones $\{\pos(\mathbf{r}_I: I \in \sigma): \sigma \in
\widetilde{\Delta}\}$ is an $(n-3)$-dimensional polyhedral fan
$\Delta$.  The associated toric variety $X_{\Delta}$ is smooth.  In
addition, there is an embedding of $\MOn$ into $X_{\Delta}$ with $\MOn
\cap T = M_{0,n}$, where $T$ is the torus of $X_{\Delta}$, and the
support of $\Delta$ is the tropical variety of $M_{0,n} \subset T$.
\end{proposition}

 An important ingredient in the proof of Proposition \ref{p:GLF} is
 the following isomorphism of Chow rings.  Note that while
 $X_{\Delta}$ is not complete, it is smooth, so there is a ring
 structure on $A^*(X_{\Delta}) = \oplus_k A^k(X_{\Delta})$, where
 $A^k(X_{\Delta}) = A_{d-k}(X_{\Delta})$ for $d={n \choose 2}-n$.

 \begin{proposition}\label{chowisomorphism} Let
$i:\MOn \rightarrow X_{\Delta}$ be
   the embedding of $\MOn$ into the toric variety
 $X_{\Delta}$ given in Proposition~\ref{p:DeltaDef}.
  The pullback
 $$i^*:A^*(X_{\Delta}) \rightarrow A^*(\MOn)$$ is an isomorphism, and
  in particular, $A_{d-(n-3)}(X_{\Delta}) \cong \mathbb{Z}$, where
  $d={n\choose2}-n$.
\end{proposition}

The second assertion of Proposition \ref{chowisomorphism} follows from
the first, but it can also be proved with an explicit toric
computation.  The only difficulty of this strategy is the
combinatorial bookkeeping.  Instead, we opt for the following
conceptual proof that relies on the realization of
$\MOn$ as a De Concini/Procesi wonderful
compactification of a particular hyperplane arrangement complement.

\begin{proof}[Proof of Proposition \ref{chowisomorphism}]
In \cite[pp533-555]{FY} Feichtner and Yuzvinsky construct a smooth toric variety
$X_{\Sigma(L,G)}$ and a ring isomorphism $\phi : A^*(X_{\Sigma(L,G)})
\rightarrow A^*(\MOn)$.  This construction is a special case of
the identification of the Chow ring of any wonderful
compactification of a hyperplane arrangement complement with the
Chow ring of a toric variety.

The fan $\Sigma(L,G)$ lies in $\mathbb R^{{n \choose 2}-n+1}$, and has
a ray for each ray of $\Delta$, plus one additional ray spanned by
$(1,\dots,1)$.  If $D'_I$ is the torus-invariant divisor corresponding
to the ray of $\Sigma(L,G)$ indexed by $I$, then $\phi(D'_I) =
\delta_I$.  A collection of rays span a cone in $\Delta$ if and only
if the corresponding collection of rays, plus the ray through
$(1,\dots,1)$, span a cone in $\Sigma(L,G)$, and the identification of
$\mathbb R^{{n \choose 2}-n}$ with $\mathbb R^{{n \choose
    2}-n+1}/\mathbb R (1,\dots, 1)$ induces a map of fans $\pi :
\Sigma(L,G) \rightarrow \Delta$.

In \cite{Oda} and \cite{Park} it is shown that the Stanley-Reisner
presentation of \cite[p106]{Fulton} describes the Chow ring of any
smooth toric variety, even if it is not complete.
The fact that $\Delta$ is the projection under $\pi$ of $\Sigma(L,G)$
implies that the Stanley-Reisner ring of $\Sigma(L,G)$ has one more
generator than that of $X_{\Delta}$, corresponding to the ray through
$(1,\dots,1)$.  It has the same monomial generators, and one more
linear relation, which involves the generator corresponding to the ray
through $(1,\dots,1)$.  The induced map $\pi^* \colon A^*(X_{\Delta})
\rightarrow A^*(X_{\Sigma(L,G)})$ is thus an isomorphism, so $i^* =
\phi \circ \pi^*$ is the desired isomorphism.

\end{proof}

We are now able to prove Proposition~\ref{p:GLF}.

\begin{proof}[Proof of Proposition \ref{p:GLF}]
We use the fact that $\widetilde{\Delta}$ is the simplicial complex
associated to $\Delta$.  It then follows from Parts~\ref{i:Gdefn} and
\ref{i:Gasintersection} of Proposition~\ref{p:Gcone} applied to the
toric variety $X_{\Delta}$ that $\mathcal{G}_{\Delta}(\MOn)=
i^*(\G)$, and from Parts~\ref{i:defnLcone} and \ref{i:modifiedGcone} of
Corollary~\ref{c:Lcone} that $\mathcal{L}_{\Delta}(\MOn) =
i^*(\L)$.  The first two inclusions then follow from
Theorem~\ref{t:bound}.
%The first inclusion then follows from
%Parts~\ref{i:Gdefn} and \ref{i:Gasintersection} of
%Proposition~\ref{p:Gcone} applied to the toric variety $\Delta$, The
%second inclusion follows from Parts~\ref{i:defnLcone} and
%\ref{i:modifiedGcone} of Corollary~\ref{c:Lcone}.  
Note that by Proposition \ref{chowisomorphism},
$A_{d-(n-3)}(X_{\Delta}) \cong \mathbb{Z}$ for $d = {n \choose 2}-n$.
Parts~\ref{i:defnUcone} and \ref{i:modifiedGconeforU} of
Proposition~\ref{p:Ucone} thus imply that $\overline{F}_{\Delta}(\MOn)
= i^*(\U)$.  The last inclusion also then follows from
Theorem~\ref{t:bound}.

We now show that the pullback of $\U$ is the cone of $F$-divisors.
Since $A_{d-(n-3)}(X_{\Delta})$ is one-dimensional, we can write $\U =
\{ D : D \cdot V(\tau) \geq 0 \text{ for all } \tau \in \Delta(n-4)
\}$.  By Proposition~\ref{chowisomorphism} there is an isomorphism
$i^* \colon A^*(X_{\Delta}) \rightarrow A^*(\MOn)$.  Now $D \cdot
  [V(\tau)] \geq 0$ if only if $i^*(D) \cdot i^*(V(\tau)) \geq 0$.
  Since $X_{\Delta}$ is smooth, $[V(\tau)] = D_{I_1}\cdot \dots \cdot
  D_{I_{n-4}}$, where $\tau$ is the cone generated by rays labelled by
  $I_1,\dots,I_{n-4}$.  Thus $i^*([V(\tau)]) = i^*(D_{I_1}) \cdot
  \dots \cdot i^*(D_{I_{n-4}}) = \delta_{I_1}\cdot \dots \cdot
  \delta_{I_{n-4}}$.  The intersection of these boundary divisors is
  the class of the $\operatorname{F}$-curve $C_{\tau}$ whose dual
  graph is the tree corresponding to $\tau$.  The set of all classes
  of $\operatorname{F}$-curves is the set of $[C_{\tau}]$ for $\tau
  \in \Delta(n-4)$, so
\begin{align*}
i^*(\U) & = \{ i^*(D) : D \cdot V(\tau) \geq 0 \text{ for all } \tau \in \Delta(n-4)\} \\
& = \{ i^*(D) : i^*(D) \cdot i^*(V(\tau)) \geq 0 \text{ for all } \tau \in \Delta(n-4) \}\\
& = \{ i^*(D) : i^*(D) \cdot [C_{\tau}] \geq 0 \text{ for all} \operatorname{F} \text{-curves }C_{\tau}\},\\
\end{align*}
which is the cone of $\operatorname{F}$-divisors.

\end{proof}

\begin{example}One can check easily by hand for $n=5$, and by computer for $n=6$, that
 $$\mathcal{G}_{\Delta}(\overline{\operatorname{M}} _{0,n}) = \mathcal{L}_{\Delta}(\overline{\operatorname{M}} _{0,n}) = \nef(\overline{\operatorname{M}} _{0,n}) 
= \mathcal{F}_{\Delta}(\overline{\operatorname{M}} _{0,n}).$$
\end{example}

One original motivation for the $\operatorname{F}$-Conjecture was the
expectation that cycles on $\MOn$ should behave like those on a toric
variety, and thus the cone of effective $k$-cycles should be generated
by classes of the $k$-dimensional strata.  This was shown to be false
for divisors ($k=n-4$) independently by Keel and Vermeire
\cite[p4]{GKM}, \cite{Vermeire}).  The $\operatorname{F}$-conjecture
is the case $k=1$.  Proposition \ref{p:GLF} enriches the connection of
$\MOn$ with toric varieties, by showing that the boundary strata are
pullbacks of torus-invariant loci of the noncomplete toric variety
$X_{\Delta}$.

In light of Proposition~\ref{p:GLF}, one way to prove the
$\operatorname{F}$-conjecture would be to show that $\L=\U$.  This has
been computationally verified for $n \leq 6$.  This suggests the following question.

\begin{question}\label{LCon} Is
$ \mathcal{L}_{\Delta}(\MOn) = \nef(\MOn)$?
\end{question}

Even if $\L = \U$, the description of this polyhedral cone given by
$\L$ may be more accessible than the facet-description given by $\U$.

Another way to find the nef cone of $\MOn$ would be to show that it is
a Mori dream space, and give generators for its Cox ring.  This has
been done for $n=6$ by Castravet~\cite{Castravet}.  The resulting
toric embedding $\MOn \rightarrow X_{\Sigma}$ is different from the
embedding into $X_{\Delta}$, as the effective cone of $X_{\Sigma}$
equals that of $\MOn$, which is strictly larger than that of
$X_\Delta$.  However for $n=6$ we have $\L = \ensuremath{\mathcal
  L_{\Sigma}} = \nef(\MOn)$.  This suggests that the phenomenon
illustrated earlier in the case of the del Pezzo surface
$\Bl_5(\mathbb P^2) \subseteq X_{\Sigma_5}$ that different toric
embeddings may give the same cone $\L$ may hold for $\MOn$.

The main goal of the $\operatorname{F}$-conjecture or of Question~\ref{LCon}
is to have a concrete description of the nef cone of $\MOn$ in order to 
study its birational geometry.   The cone
 $\L$  would give such a description.

\raggedright

\end{document}